\documentclass[reqno, 11pt]{amsart}
\usepackage{amsmath, amsthm, amscd, amsfonts, amssymb, graphicx, color, mathrsfs}
\usepackage{extarrows}
\usepackage[bookmarksnumbered, colorlinks, plainpages]{hyperref}
\hypersetup{colorlinks=true,linkcolor=red, anchorcolor=green,
	citecolor=cyan, urlcolor=red, filecolor=magenta, pdftoolbar=true}

\newtheorem{theorem}{Theorem}[section]
\newtheorem{lemma}[theorem]{Lemma}
\newtheorem{proposition}[theorem]{Proposition}
\newtheorem{corollary}[theorem]{Corollary}
\theoremstyle{definition}
\newtheorem{definition}[theorem]{Definition}
\newtheorem{example}[theorem]{Example}

\theoremstyle{remark}
\newtheorem{remark}[theorem]{Remark}
\numberwithin{equation}{section}

\begin{document}
	
	\setcounter{page}{1}
	
	\title[Power mean transforms of operators]
	{Power mean transforms of operators}

	\author[J.B. Zhou]{Jing-Bin Zhou}
	\address{Jing-Bin Zhou, School of Mathematics,
		Shanghai University of Finance and Economics,
		777 Gouding Road, Shanghai 200433, P. R. China}
	\email{zhoujingbin@stu.sufe.edu.cn}
	
	\author[S. Yang]{Shihai Yang}
	\address{Shihai Yang, School of Mathematics,
		Shanghai University of Finance and Economics,
		777 Gouding Road, Shanghai 200433, P. R. China}
	\email{yang.shihai@mail.shufe.edu.cn}

	
	\begin{abstract}
		In this paper, we introduce the power mean transform $P_{\lambda}(T)$ of an operator $T$ on a Hilbert space, which is a convex combination of some classical operator transforms such as the mean transform $M(T)$, the Aluthge transform $\Delta(T)$, and the Duggal transform $T^D$. In particular, when 
		$T$ is invertible, this transform coincides with the induced Aluthge transform $\Delta_{\mathsf{m}_{f}}(T)$ recently defined by Yamazaki \cite{yamazaki-laa-2021} with $f(x)=(\lambda+(1-\lambda)\sqrt{x})^2$ for $x\in(0,\infty)$ and $\lambda\in(0,1)$. We study basic properties of $P_{\lambda}(T)$ including its spectrum, norm and numerical radius. Moreover, we use the power mean transform to give new characterizations of normal, quasinormal and binormal operators. The questions of Golla et al. \cite{yamazaki-laa-2023} and some new results on the Duggal transform are also mentioned. We obtain a result close to the recent one of Osaka and Yamazaki \cite[Theorem 3.3]{yamazaki-tams-2025} on the iteration of the induced Aluthge transform for centered operators. Finally, we describe the form of bijective maps commuting with the power mean transform of the product of matrices.
	\end{abstract}
	\maketitle
	\section{Introduction}
	Let $\mathbb{B}(\mathcal{H})$ denote the $C^\ast$-algebra of all bounded linear operators on the Hilbert space $\mathcal{H}$ with the inner product $\langle\cdot,\cdot\rangle$. Then the norm of $x\in\mathcal{H}$ is given by $\|x\|=\langle x,x\rangle^{\frac{1}{2}}$. For an operator $T\in\mathbb{B}(\mathcal{H})$, denote its image space, null space, adjoint, and norm by $\mathcal{R}(T)$, $\mathcal{N}(T)$, $T^\ast$ and $\|T\|$, respectively. If $\langle Tx,x\rangle\ge0$ for all $x\in\mathcal{H}$, then $T$ is called a positive operator, and we write its square root as $T^\frac{1}{2}$. It is well known that every operator $T\in\mathbb{B}(\mathcal{H})$ admits a unique polar decomposition $T=U|T|$, where $U$ is  a partial isometry with $\mathcal{N}(U)=\mathcal{N}(T)$ and  $|T|=(T^\ast T)^\frac{1}{2}$. Now, let $T=U|T|$ be the polar decomposition. In \cite{Foias-PJM-2003}, Foias et al. introduced the {\em Duggal transform} as $T^D=|T|U$ and also examined the {\em Aluthge transform} $\Delta(T)=|T|^\frac{1}{2}U|T|^\frac{1}{2}$ (which was first studied in \cite{Aluthge-IEOT-1990}). The latter was generalized in \cite{Cho-Tanahashi-SMJ-2002} by Ch\={o} and Tanahashi as $\Delta_{\lambda}(T)=|T|^{\lambda}U|T|^{1-\lambda}$ for $\lambda\in[0,1]$. In particular, $\Delta_{0}(T)=U|T|=T$ and $\Delta_{1}(T)=|T|U=T^D$. For the sequel, let $l^2$ denote the Hilbert space of square summable sequences with standard orthonormal basis $\{e_{n}\}_{n=0}^\infty$. Let $\alpha\equiv\{\alpha_{n}\}_{n=0}^\infty$ be a bounded sequence of positive numbers (called weights). The unilateral weighted shift $W_{\alpha}$ on $l^2$ is defined by $W_{\alpha}e_{n}=\alpha_{n}e_{n+1}$. This operator is also denoted by $\operatorname{shift}(\alpha_{0},\alpha_{1},\alpha_{2},\cdots)$, a notation that directly indicates the weights \cite[Problem 89]{Halmos-GTM-1982}. The polar decomposition of $W_{\alpha}$ is $W_{\alpha}=U|W_{\alpha}|$, where $U=\operatorname{shift}(1,1,1,\cdots)$ and $|W_{\alpha}|e_{n}=\alpha_{n}e_{n}$. In \cite{Lee-JMAA-2014}, Lee et al. mentioned that $\Delta(W_{\alpha})=\operatorname{shift}(\sqrt{\alpha_{0}\alpha_{1}},\sqrt{\alpha_{1}\alpha_{2}},\cdots)$, and noted that the convex combination of $\Delta_{0}(W_{\alpha})$ and $\Delta_{1}(W_{\alpha})$ yields $$\frac{1}{2}\Delta_{0}(W_{\alpha})+\frac{1}{2}\Delta_{1}(W_{\alpha})=\frac{1}{2}W_{\alpha}+\frac{1}{2}W_{\alpha}^D=\operatorname{shift}(\frac{\alpha_{0}+\alpha_{1}}{2},\frac{\alpha_{1}+\alpha_{2}}{2},\cdots).$$ Here, the entries $\sqrt{\alpha_{n}\alpha_{n+1}}$ and $\frac{\alpha_{n}+\alpha_{n+1}}{2}$ are the geometric and arithmetic means of the consecutive weights $\alpha_{n}$ and $\alpha_{n+1}$ for $n\ge0$, respectively. Following this observation, the authors \cite{Lee-JMAA-2014} defined the {\em mean transform} for a general operator $T\in\mathcal{B}(H)$ as the convex combination
	\begin{equation}\label{mean}
	M(T)\equiv\frac{1}{2}\Delta_{0}(T)+\frac{1}{2}\Delta_{1}(T)=\frac{1}{2}T+\frac{1}{2}T^D.
	\end{equation}  Benhida et al. \cite{Benhida-BJMA-2020} extended this transform to
	\begin{equation}\label{heinz1}
		H_{\lambda}(T)\equiv\frac{1}{2}\Delta_{\lambda}(T)+\frac{1}{2}\Delta_{1-\lambda}(T)
	\end{equation}
	for $\lambda\in[0,1]$, which Stanković recently called the {\em Heinz transform} of $T$ \cite{Stankovic-ZAA-2024}. We see that
	\begin{equation}\label{heniz}
		H_{\lambda}(W_{\alpha})=\operatorname{shift}(\frac{\alpha_{0}^{1-\lambda}\alpha_{1}^{\lambda}+\alpha_{0}^\lambda\alpha_{1}^{1-\lambda}}{2},\frac{\alpha_{1}^{1-\lambda}\alpha_{2}^{\lambda}+\alpha_{1}^\lambda\alpha_{2}^{1-\lambda}}{2},\cdots),
	\end{equation}
	where $\frac{\alpha_{n}^{1-\lambda}\alpha_{n+1}^{\lambda}+\alpha_{n}^\lambda\alpha_{n+1}^{1-\lambda}}{2}$ is the Heinz mean of the consecutive weights $\alpha_{n}$ and $\alpha_{n+1}$ for $n\ge0$. Indeed, a direct computation via the polar decomposition of $W_{\alpha}$ yields $\Delta_{\lambda}(W_{\alpha})e_{n}=\alpha_{n}^{1-\lambda}\alpha_{n+1}^{\lambda}e_{n+1}$, i.e., $\Delta_{\lambda}(W_{\alpha})=\operatorname{shift}(\alpha_{0}^{1-\lambda}\alpha_{1}^\lambda,\alpha_{1}^{1-\lambda}\alpha_{2}^\lambda,\cdots)$. Consequently, \eqref{heniz} holds by the definition of $H_{\lambda}(W_{\alpha})$. On the other hand, the convex combination of $M(W_{\alpha})$ and $\Delta(W_{\alpha})$ yields
	$$P(W_{\alpha})\equiv\frac{1}{2}M(W_{\alpha})+\frac{1}{2}\Delta(W_{\alpha})=\operatorname{shift}(\frac{1}{2}\sqrt{\alpha_{0}\alpha_{1}}+\frac{1}{2}\frac{\alpha_{0}+\alpha_{1}}{2},\frac{1}{2}\sqrt{\alpha_{1}\alpha_{2}}+\frac{1}{2}\frac{\alpha_{1}+\alpha_{2}}{2},\cdots),$$
	where $\frac{1}{2}\sqrt{\alpha_{n}\alpha_{n+1}}+\frac{1}{2}\frac{\alpha_{n}+\alpha_{n+1}}{2}=(\frac{1}{2}\sqrt{\alpha_{n}}+\frac{1}{2}\sqrt{\alpha_{n+1}})^2$ is the Heron mean, and also the power mean, of $\alpha_{n}$ and $\alpha_{n+1}$ \cite{Ito-JMI-2020}. In view of these results, we call the operator
	\begin{equation*}
P(T)\equiv\frac{1}{2}M(T)+\frac{1}{2}\Delta(T)
	\end{equation*}
	 the {\em Heron mean transform} of $T\in\mathbb{B}(\mathcal{H})$. Moreover, recently, \eqref{mean} and \eqref{heinz1} were extended in \cite{zamani-jmaa-2021} and \cite{feki-rm-2025} to the more general operator transforms $$M_{\lambda}(T)\equiv\lambda\Delta_{0}(T)+(1-\lambda)\Delta_{1}(T)=\lambda T+(1-\lambda)T^D$$and$$H_{\lambda_{1},\lambda_{2}}(T)=\lambda_{2}\Delta_{\lambda_{1}}(T)+(1-\lambda_{2})\Delta_{1-\lambda_{1}}(T),$$where $\lambda,\lambda_{1},\lambda_{2}\in[0,1]$. The weights of the weighted shifts $M_{\lambda}(W_{\alpha})$ and $H_{\lambda_{1},\lambda_{2}}(W_{\alpha})$ are, respectively, the weighted arithmetic mean and the weighted Heinz mean of $\alpha_{n}$ and $\alpha_{n+1}$. Furthermore, we note that the following convex combination of $M_{\frac{(1-\lambda)^2}{(1-\lambda)^2+\lambda^2}}(W_{\alpha})$ and $\Delta(W_{\alpha})$ gives
	 \begin{equation}\label{uni}
	 \begin{aligned}
 P_{\lambda}(W_{\alpha})&\equiv(1-2\lambda+2\lambda^2)M_{\frac{(1-\lambda)^2}{(1-\lambda)^2+\lambda^2}}(W_{\alpha})+(2\lambda-2\lambda^2)\Delta(W_{\alpha})\\&=(1-\lambda)^2W_{\alpha}+\lambda^2W_{\alpha}^D+2\lambda(1-\lambda)\Delta(W_{\alpha})\\&=\operatorname{shift}\Big(((1-\lambda)\sqrt{\alpha_{0}}+\lambda\sqrt{\alpha_{1}})^2,((1-\lambda)\sqrt{\alpha_{1}}+\lambda\sqrt{\alpha_{2}})^2,\cdots\Big)
	 \end{aligned}
	 \end{equation}
	 for $\lambda\in[0,1]$, where $((1-\lambda)\sqrt{\alpha_{n}}+\lambda\sqrt{\alpha_{n+1}})^2$ is the weighted power mean of $\alpha_{n}$ and $\alpha_{n+1}$. Inspired by these facts, we have the following. \begin{definition}
	  Let $\lambda\in[0,1]$. The {\em power mean transform} $P_{\lambda}(T)$ of $T\in\mathbb{B}(\mathcal{H})$ is defined as $$P_{\lambda}(T)=(1-\lambda)^2 T+\lambda^2 T^{D}+2\lambda(1-\lambda)\Delta(T).$$
	 	In particular, $P_{0}(T)=T$, $P_{1}(T)=T^D$ and $P_{\frac{1}{2}}(T)=P(T)$.
	 \end{definition}
	 It should be mentioned that Yamazaki introduced the operator transformation below, which is closely related to the power mean transform.
	
	\begin{definition}\cite{yamazaki-laa-2021, yamazaki-tams-2025}
	Let $T=U|T|\in\mathbb{B}(\mathcal{H})$ with the spectral decomposition $|T|=\int_{\sigma(T)}sdE_{s}$. For an operator mean $\mathsf{m}_{f}$ with a representing function $f$, the {\em induced Aluthge transformation} $\Delta_{\mathsf{m}_{f}}(T)$ of $T$ is definde as follows.
	\begin{itemize}
		\item[(i)] If $|T|$ is invertible, then $$\Delta_{\mathsf{m}_{f}}(T)=\int_{\sigma(|T|)}\int_{\sigma(|T|)}sf\Big(\frac{t}{s}\Big)dE_{s}UdE_{t}.$$
		\item[(ii)] If $|T|$ is not invertible, and if there exists an isometry $V$ such that $T_{\varepsilon}=V(|T|+\varepsilon I)$ is the polar decomposition for all $\varepsilon>0$ and $\lim\limits_{\varepsilon\searrow0}T_{\varepsilon}=T$, then $$\Delta_{\mathsf{m}_{f}}(T)=\lim\limits_{\varepsilon\searrow0}\Delta_{\mathsf{m}_{f}}(T_{\varepsilon}),$$where the "$\lim$" means the strong operator topology. 
	\end{itemize}
	\end{definition}
	For $\lambda\in[0,1]$, define $f(x)=(\lambda+(1-\lambda)\sqrt{x})^2$ on $(0,\infty)$. Then for invertible $T\in\mathbb{B}(\mathcal{H})$, the induced Aluthge transform $\Delta_{\mathsf{m}_{f}}(T)$ coincides with the power mean transform $P_{\lambda}(T)$ (see \cite[Example 1.2 (iii)]{yamazaki-tams-2025} and \cite[Example 1.8]{yamazaki-laa-2023}). However, for non‑invertible $T$, the identity 
	$P_{\lambda}(T)=\Delta_{\mathsf{m}_{f}}(T)$ does not hold in general. Indeed, consider the matrix $T= \begin{bmatrix}
		0&1\\
		0&0
	\end{bmatrix}$. It has the unique polar decomposition $T=U|T|$ with $U=\begin{bmatrix}
		0&1\\
		0&0
	\end{bmatrix}$ and $|T|=\begin{bmatrix}
		0&0\\
		0&1
	\end{bmatrix}$. By the equivalence of the norm topology and the strong operator topology in finite-dimensional spaces, one can easily verify that $T_{\varepsilon}=\begin{bmatrix}
	0&1+\varepsilon\\
	\varepsilon&0
	\end{bmatrix}=V(|T|+\varepsilon I)$ is the polar decomposition for all $\varepsilon>0$ and $\lim\limits_{\varepsilon\searrow 0} T_{\varepsilon}=T$, where $V=\begin{bmatrix}
	0&1\\
	1&0
	\end{bmatrix}$. A direct computation shows that
	
\begin{align*}
\Delta_{\mathsf{m}_{f}}(T)=\lim\limits_{\varepsilon\searrow0}\Delta_{\mathsf{m}_{f}}(T_{\varepsilon})=\lim\limits_{\varepsilon\searrow0}P_{\lambda}(T_{\varepsilon})&=\lim\limits_{\varepsilon\searrow0}\begin{bmatrix}
	0&((1-\lambda)\sqrt{1+\varepsilon}+\lambda\sqrt{\varepsilon})^2\\((1-\lambda)\sqrt{\varepsilon}+\lambda\sqrt{1+\varepsilon})^2&0
\end{bmatrix}\\&=\begin{bmatrix}
0&(1-\lambda)^2\\\lambda^2&0
\end{bmatrix}\ne\begin{bmatrix}
0&(1-\lambda)^2\\0&0
\end{bmatrix}=P_{\lambda}(T).
\end{align*}
Hence, for any invertible $T\in\mathbb{B}(\mathcal{H})$ and $\lambda\in[0,1]$, our results hold for $\Delta_{\mathsf{m}_{f}}(T)$ with $f(x)=(\lambda+(1-\lambda)\sqrt{x})^2$ on $(0,\infty)$.

Let $\sigma(T)$ denote the spectrum of $T$, and let $W(T)$ be its numerical range, i.e., $W(T)=\{\langle Tx,x\rangle:\|x\|=1\}$. The {\em spectral radius} $r(T)$ and the {\em numerical radius} $\omega(T)$ of $T$ are defined by $r(T)=\sup\{|\lambda|:\lambda\in\sigma(T)\}$ and $\omega(T)=\sup\{|\lambda|:\lambda\in W(T)\}$, respectively. It is well known that map $\omega(\cdot):T\mapsto\omega(T)$ is a norm on $\mathbb{B}(\mathcal{H})$. These quantities satisfy the inequality
\begin{equation}\label{spnuno}
	r(T)\le\omega(T)\le\|T\|.
\end{equation}

In Section 2, we characterize the invertibility of $T$ by the power mean transform, discuss the relation between $\sigma(T)$ and $\sigma(P_{\lambda}(T))$, and estimate $\|P_{\lambda}(T)\|$ via Heinz's inequality. As corollaries, we improve the Cauchy-Schwarz inequality and give numerical range conditions for certain equalities. We also extend a recent result of Feki and Yamazaki \cite[Theorem 2]{feki-mia-2021} on the numerical radius.

For $T,S\in\mathbb{B}(\mathcal{H})$, the commutator is defined by $[T,S]=TS-ST$. An operator $T\in\mathbb{B}(\mathcal{H})$ is called {\em normal} if \([T,T^*]=0\); {\em quasinormal} if \([T,T^*T]=0\); {\em centered} if \([T^n(T^n)^*, (T^m)^*T^m]=0\) for $n,m\in\{1,2,\cdots\}$; and {\em binormal} if \([T^*T,TT^*]=0\).  
Let $T=U|T|$be the polar decomposition. It is well known that \(T\) is quasinormal if and only if \([U,|T|]=0\) \cite[Lemma 2.1]{curto-aof-2019}.  
Campbell \cite{cam-pams-1972} proved that \(T\) is binormal if and only if \([|T|, U|T|U^*]=0\). Ito et al. \cite[Theorem 4.1]{Ito-JOT-2004} showed that \(T\) is centered if and only if \([U^n|T|(U^n)^*, |T|]=0\) for $n=1,2,\cdots$.  
Consequently, we have the inclusions  
\[
\text{normal} \subseteq \text{quasinormal} \subseteq \text{centered} \subseteq \text{binormal}.
\]  
Using the polar decomposition of weighted shifts, one readily sees that every weighted shift is central.

In Section 3, we use the power mean transform to give new characterizations of normal, quasinormal and binormal operators. In particular, we also obtain some new results concerning the Duggal transform. In addition, we mention the questions raised by Golla et al. \cite{yamazaki-laa-2023}. Moreover, we prove that the iteration of the power mean transform of a semi‑hyponormal centered operator converges to a quasinormal operator; this is close to a recent result of Osaka and Yamazaki \cite[Theorem 3.3]{yamazaki-tams-2025}.

The preserver problem for operator transforms has also attracted attention. In this direction, the first result was obtained by Botelho et al. in \cite{mol-blm-2016}, where they characterized the linear maps that commute with the $\lambda$-Aluthge transform. Later, Chabbabi \cite{chabbabi-jmaa-2017} considered the preserver problem for nonlinear maps that commute with the $\lambda$-Aluthge transform. Recently, Bourhim et al. have also worked on the unitary‑similarity preserver problem for operator transforms (see \cite{bour-lma-2026, bour-jmaa-2025, bour-lma-2025}).

In the last seciton, following the recent ideas of Bourhim et al. \cite{bour-lma-2026}, we describe the form of all bijective maps $\Phi$ on $\mathbb{B}(\mathcal{H})$ that satisfy
$$P_{\lambda}(\Phi(T)\Phi(S))=\Phi(P_{\lambda}(TS))\qquad\text{for all $T,S\in\mathbb{B}(\mathcal{H})$}$$ in the matrix case.
	\section{Basic properties of the power mean transform $P_{\lambda}(T)$}
	 We start this section with some notations and facts. Denote by $\overline{A}$ the closure of a subset $A$ in a metric space, by $P_{M}$ the orthogonal projection onto the closed subspace $M$, and by $N^\perp$ the orthogonal complement of the subspace $N$. Let $T=U|T|$ be the polar decomposition of $T$. Then we have$$\mathcal{N}(T)=\mathcal{N}(|T|)=\mathcal{N}(U)\quad\text{and}\quad\mathcal{N}(T^\ast)=\mathcal{N}(|T^\ast|)=\mathcal{N}(U^\ast)$$and$$\overline{\mathcal{R}(T^\ast)}=\overline{\mathcal{R}(|T|)}=\mathcal{R}(U^\ast)\quad\text{and}\quad\overline{\mathcal{R}(T)}=\overline{\mathcal{R}(|T^\ast|)}=\mathcal{R}(U).$$ Combining the above relations with $U^\ast U=P_{\overline{\mathcal{R}(T^\ast)}}$ and $UU^\ast=P_{\overline{\mathcal{R}(T)}}$ immediately yields 
	 \begin{equation}\label{polarbasic}
	 	U^\ast UT^\ast=T^\ast,\ U^\ast U|T|=|T|,\ U^\ast UU^\ast=U^\ast,\ UU^\ast T=T\ \text{and}\ UU^\ast|T^\ast|=|T^\ast|.
	 \end{equation}
	On the other hand, closely related to the unilateral weighted shift operator is the bilateral weighted shift operator. In what follows, let $\mathbb{R}$, $\mathbb{C}$, $\mathbb{N}$, and $\mathbb{Z}$ denote the sets of real numbers, complex numbers, non-negative integers and integers, respectively. Let $\alpha\equiv\{\alpha_{n}\}_{n\in\mathbb{Z}}$ be a bounded sequence of positive numbers. The bilateral weighted shift $W_{\alpha}$ on $l^2(\mathbb{Z})$ is defined by $W_{\alpha}e_{n}=\alpha_{n}e_{n+1}$ for $n\in\mathbb{Z}$, where $\{e_{n}\}_{n\in\mathbb{Z}}$ is the standard orthonormal basis. 
It is well known that $W_{\alpha}$ is invertible if and only if $\inf\limits_{n\in\mathbb{Z}}\alpha_{n}>0$. Furthermore, similar to the calculation in \eqref{uni}, we have
\begin{equation}\label{meanuni}
	P_{\lambda}(W_{\alpha})e_{n}=((1-\lambda)\sqrt{\alpha_{n}}+\lambda\sqrt{\alpha_{n+1}})^2e_{n+1}
\end{equation}
for $\lambda\in[0,1]$. Then we have the following.
	\begin{proposition}\label{inws}
		Let $\lambda\in[0,1]$. Let $\alpha\equiv\{\alpha_{n}\}_{n\in\mathbb{Z}}$ be a bounded sequence of positive numbers and $W_{\alpha}$ the bilateral weighted shift. If $W_{\alpha}$ is invertible, then so is $P_{\lambda}(W_{\alpha})$.
	\end{proposition}
\begin{proof}
If $\lambda=1$, then $P_{\lambda}(W_{\alpha})=W_{\alpha}^D$. Applying the fact that $T\in\mathbb{B}(\mathcal{H})$ is invertible if and only if $P_{1}(T)=T^D$, we obtain the desired result directly. For $\lambda\in[0,1)$, from \eqref{meanuni} we know that $P_{\lambda}(W_{\alpha})$ is a bilateral weighted shift. It suffices to prove $\inf\limits_{n\in\mathbb{Z}}((1-\lambda)\sqrt{\alpha_{n}}+\lambda\sqrt{\alpha_{n+1}})^2>0$. This follows immediately from $\inf\limits_{n\in\mathbb{Z}}\alpha_{n}>0$ and $(1-\lambda)^2\alpha_{n}\le((1-\lambda)\sqrt{\alpha_{n}}+\lambda\sqrt{\alpha_{n+1}})^2$ for $n\in\mathbb{Z}$.
\end{proof}
\begin{remark}\label{nonclosed}
For $\lambda\in[0,1)$, the converse of this proposition is not true in general. For example, let $\alpha_{n}=\begin{cases}
	1&\text{$n$ is even}\\
	\frac{1}{n^2}&\text{$n$ is odd}
\end{cases}.$ A direct computation gives $P_{\lambda}(W_{\alpha})e_{n}=\widetilde{\alpha_{n}}e_{n+1}$, where $\widetilde{\alpha_{n}}=\begin{cases}
(1-\lambda+\frac{\lambda}{|n+1|})^2&\text{$n$ is even}\\
(\frac{1-\lambda}{|n|}+\lambda)^2&\text{$n$ is odd}
\end{cases}$. From the inequality $\max\{\lambda^2,(1-\lambda)^2\}\le\widetilde{\alpha_{n}}\le1$ ($n\in\mathbb{Z}$) we obtain that $P_{\lambda}(W_{\alpha})$ is invertible. However, $\inf\limits_{n\in\mathbb{Z}}\alpha_{n}=0$ implies that $W_{\alpha}$ is not invertible.
\end{remark} 
In fact, the invertibility of $T\in\mathbb{B}(\mathcal{H})$ can be characterized via the power mean transform of $T$. To present this result, we first have the following proposition.
\begin{proposition}\label{kernel}
	Let $\lambda\in[0,1)$ and $T\in\mathbb{B}(\mathcal{H})$. Then $\mathcal{N}(P_{\lambda}(T))=\mathcal{N}(T)$. In particular $P_{\lambda}(T)=0$ if and only if $T=0$.
\end{proposition}
\begin{proof}
	Let $T=U|T|$ be the polar decomposition of $T$. Suppose $x\in\mathcal{N}(P_{\lambda}(T))$. Write $u=(1-\lambda)|T|^\frac{1}{2}x$ and $v=\lambda|T|^\frac{1}{2}Ux$. We obtain 
	\begin{equation}\label{iso}
		\|Uu\|=\|u\|
	\end{equation}
	since $U$ is an isometry on $(\mathcal{N}(U))^{\perp}=(\mathcal{N}(|T|))^{\perp}=(\mathcal{N}(|T|^\frac{1}{2}))^\perp=\overline{\mathcal{R}(|T|^{\frac{1}{2}})}$. On the other hand, from the definition of $P_{\lambda}(T)$ and the second equation in \eqref{polarbasic}, we have
\begin{equation}\label{equ}
\langle(1-\lambda)^2|T|x,x\rangle+\langle\lambda^2U^\ast|T|Ux,x\rangle+\langle 2\lambda(1-\lambda)U^\ast |T|^\frac{1}{2}U|T|^\frac{1}{2}x,x\rangle=\langle U^\ast P_{\lambda}(T)x,x\rangle=0,
\end{equation}
	which implies, by the positivity of $(1-\lambda)^2|T|$ and $\lambda^2U^\ast|T|U$, that $\langle 2\lambda(1-\lambda)U^\ast |T|^\frac{1}{2}U|T|^\frac{1}{2}x,x\rangle$ is a real number. Therefore, using \eqref{iso}, the above equation can be rewritten as
	$$\|u\|^2+\|v\|^2+2\operatorname{Re}\langle \lambda(1-\lambda)U^\ast |T|^\frac{1}{2}U|T|^\frac{1}{2}x,x\rangle=\|Uu\|^2+\|v\|^2+2\operatorname{Re}\langle Uu,v\rangle=\|Uu+v\|^2=0,$$
		where $\operatorname{Re}\langle y,z\rangle$ is the real part of $\langle y,z\rangle$ for $y,z\in\mathcal{H}$. The last equation yields $Uu=-v$, i.e., $$(1-\lambda)U|T|^\frac{1}{2}x=-\lambda|T|^\frac{1}{2}Ux.$$ Further, together with $U^\ast U=P_{\overline{\mathcal{R}(|T|)}}=P_{\overline{\mathcal{R}(|T|^\frac{1}{2})}}$, left-multiplying the above equation by $U^\ast$ gives $$\langle(1-\lambda)|T|^\frac{1}{2}x,x\rangle=\langle(1-\lambda)U^\ast U|T|^\frac{1}{2}x,x\rangle=-\langle\lambda U^\ast|T|^\frac{1}{2}Ux,x\rangle.$$ By the positivity of $(1-\lambda)|T|^\frac{1}{2}$ and $\lambda U^\ast|T|^\frac{1}{2}U$, the above equation implies $\langle(1-\lambda)|T|^\frac{1}{2}x,x\rangle=0$, and hence $\sqrt{1-\lambda}|T|^\frac{1}{4}x=0$. Consequently, $x\in\mathcal{N}(|T|^\frac{1}{4})=\mathcal{N}(|T|)=\mathcal{N}(T).$ Therefore, $\mathcal{N}(P_{\lambda}(T))\subseteq\mathcal{N}(T)$. The reverse inclusion $\mathcal{N}(T)\subseteq\mathcal{N}(P_{\lambda}(T))$ follows immediately from $\mathcal{N}(T)=\mathcal{N}(|T|^\frac{1}{2})=\mathcal{N}(|T|)=\mathcal{N}(U)$ and the definition of $P_{\lambda}(T)$.
\end{proof}
\begin{remark}\label{duggal}
	If $\lambda=1$, then $P_{\lambda}(T)=T^D$ for $T\in\mathbb{B}(\mathcal{H})$; however, $\mathcal{N}(T^D)=\mathcal{N}(T)$ does not hold in general. Consider the matrix $T= \begin{bmatrix}
		0&1\\
		0&0
	\end{bmatrix}=U|T|$ with $U=\begin{bmatrix}
		0&1\\
		0&0
	\end{bmatrix}$ and $|T|=\begin{bmatrix}
		0&0\\
		0&1
	\end{bmatrix}$. Then  $T^D=|T|U=\begin{bmatrix}
		0&0\\
		0&0
	\end{bmatrix}$ and hence $\mathcal{N}(T^D)\ne\mathcal{N}(T)$.
\end{remark}
\begin{theorem}\label{invertiblepmt}
	Let $\lambda\in(0,1)$ and $T\in\mathbb{B}(\mathcal{H})$. Then $T$ is invertible if and only if $P_{\lambda}(T)$ is invertible and $\mathcal{R}(T)$ is closed.
\end{theorem}
\begin{proof}
	Let $T=U|T|$ be the polar decomposition of $T$. \textit{Sufficiency}. Write $A=|T|^\frac{1}{2}$ and $B=U^\ast|T|^\frac{1}{2}U$. If $T$ is invertible, then it follows immediately that $\mathcal{R}(T)$ is closed and $U$ is unitary. Consequently, one can readily verify that $P_{\lambda}(T)=U((1-\lambda)^2A^2+\lambda^2B^2+2\lambda(1-\lambda)BA)$. Then, the proof reduces to showing that $C\triangleq(1-\lambda)^2A^2+\lambda^2B^2+2\lambda(1-\lambda)BA$ is invertible, i.e., $C$ and $C^\ast$ are bounded below. Indeed, the invertibility of $T$ implies that the positive operator $A^\frac{1}{2}$ (resp. $B^\frac{1}{2}$) is invertible and hence bounded below, that is, there exists $\delta_{1}>0$ (resp. $\delta_{2}>0$) such that $\|A^\frac{1}{2}x\|\ge\delta_{1}\|x\|$ (resp. $\|B^\frac{1}{2}x\|\ge\delta_{2}\|x\|$) for all $x\in\mathcal{H}$. Now we have
	\begin{align*}
		\Big\|\Big(((1-\lambda)A+\lambda B)^\frac{1}{2}\Big)x\Big\|^2&=\langle ((1-\lambda)A+\lambda B)x,x\rangle=(1-\lambda)\langle Ax,x\rangle+\lambda\langle Bx,x\rangle\\&=(1-\lambda)\|A^\frac{1}{2}x\|^2+\lambda\|B^\frac{1}{2}x\|^2\ge((1-\lambda)\delta_{1}^2+\lambda\delta_{2}^2)\|x\|^2,
	\end{align*}which implies the positive operator $((1-\lambda)A+\lambda B)^\frac{1}{2}$ is invertible, and hence $(1-\lambda)A+\lambda B$ is also invertible. Equivalently, $\|((1-\lambda)A+\lambda B)x\|\ge\delta_{3}\|x\|$ holds for some $\delta_{3}>0$ and every $x\in\mathcal{H}$. Then from the above inequality and
	$$\frac{C+C^\ast}{2}=(1-\lambda)^2A^2+\lambda^2B^2+\lambda(1-\lambda)BA+\lambda(1-\lambda)AB=((1-\lambda)A+\lambda B)^2$$ we deduce
	\begin{equation}\label{bounded}
		|\langle Cx,x\rangle|\ge\frac{1}{2}(\langle Cx,x\rangle+\overline{\langle Cx,x\rangle})=\langle \frac{C+C^\ast}{2}x,x\rangle=\|((1-\lambda)A+\lambda B)x\|^2\ge\delta_{3}^2\|x\|^2.
	\end{equation}
	Finally, it follows from the Cauchy-Schwarz inequality and \eqref{bounded} that $$\|Cx\|\|x\|\ge|\langle Cx,x\rangle|\ge\delta_{3}^2\|x\|^2$$
	and
	$\|C^\ast x\|\|x\|\ge|\langle C^\ast x,x\rangle|=|\langle Cx,x\rangle|\ge\delta_{3}^2\|x\|^2$, which means $C$ and $C^\ast$ are bounded below. Then the sufficiency is proved.
	
	\textit{Necessity}. Suppose $P_{\lambda}(T)$ is invertible. From Proposition \ref{kernel} we have $\mathcal{N}(T)=0$, which implies, by the closedness of $\mathcal{R}(T)$, that $T$ is bounded below. Then, since $\||T|x\|^2=\langle |T|x,|T|x\rangle=\langle |T|^2x,x\rangle=\langle T^\ast Tx,x\rangle=\|Tx\|^2$, it follows that $|T|$ is bounded below. Hence, $|T|$ is invertible, so
	\begin{equation}\label{sur}
		\mathcal{R}(U^\ast)=(\mathcal{N}(U))^\perp=(\mathcal{N}(|T|))^\perp=\mathcal{H}.
	\end{equation}
	Now, if we can prove that $U^\ast P_{\lambda}(T)$ is invertible, then \begin{equation*}
		U^\ast=U^\ast \Big(P_{\lambda}(T)(P_{\lambda}(T))^{-1}\Big)=\Big(U^\ast P_{\lambda}(T)\Big)(P_{\lambda}(T))^{-1}
	\end{equation*}
	expresses $U^\ast$ as a product of invertible operators, so $U^\ast$ is invertible. Equivalently, $U$ is invertible and consequently $T=U|T|$ is the product of invertible operators, i.e., $T$ is invertible. Indeed, from \eqref{sur} and the invertibility of $P_{\lambda}(T)$ we have that $U^\ast P_{\lambda}(T)$ is surjective. To show that $U^\ast P_{\lambda}(T)$ is injective, take $x\in\mathcal{N}(U^\ast P_{\lambda}(T))$. Consider the equation $\langle U^\ast P_{\lambda}(T)x,x\rangle=0$. Then, using the same argument as in Proposition \ref{kernel} for \eqref{equ}, we obtain $x\in\mathcal{N}(|T|)$. The invertibility of $|T|$ implies $x=0$. Hence, $U^\ast P_{\lambda}(T)$ is injective. The proof is complete. 	
\end{proof}
\begin{remark}\label{invpmt1}
	(i) If $\lambda=1$, then it is well-known that $T\in\mathbb{B}(\mathcal{H})$ is invertible if and only if $P_{\lambda}(T)=P_{1}(T)=T^D$ is invertible. If $\lambda=0$, then $P_{\lambda}(T)=P_{0}(T)=T$; hence this case is trivial.
	
	(ii) By Remark \ref{nonclosed}, the closedness of $\mathcal{R}(T)$ in the above theorem cannot be removed (since $\inf\limits_{n\in\mathbb{Z}}\alpha_{n}=0$, $\mathcal{R}(W_{\alpha})$ is not closed).
\end{remark}

For $\lambda\in\{0,1\}$, we have $\sigma(T)=\sigma(P_{\lambda}(T))$. The above theorem implies $\sigma(T)\not\subseteq\sigma(P_{\lambda}(T))$ for $\lambda\in(0,1)$ in general. However, if $T$ is quasinormal, i.e., $T=T^D$, then we immediately obtain $P_{\lambda}(T)=T$ since $\Delta(T)=T$, which yields $\sigma(T)=\sigma(P_{\lambda}(T))$. In fact, this result also holds for partial isometries.
\begin{corollary}
Let $\lambda\in[0,1]$ and $U\in\mathbb{B}(\mathcal{H})$ be a partial isometry. Then $\sigma(U)=\sigma(P_{\lambda}(U))$.
\end{corollary}
\begin{proof}
It is well known that $U=U|U|$ is the polar decomposition of $U$ (with the orthogonal projection $|U|=U^\ast U$). Then, from the definition of $P_{\lambda}(U)$ and the third equation in \eqref{polarbasic}, we have $$P_{\lambda}(U)=(1-\lambda)^2U+\lambda^2U^\ast UU+2\lambda(1-\lambda)U^\ast UUU^\ast U=((1-\lambda)^2+(2\lambda-\lambda^2)U^\ast U)U.$$
By \eqref{polarbasic} and Jacobson's Lemma, we obtain
$$\sigma(P_{\lambda}(U))\backslash\{0\}=\sigma(U((1-\lambda)^2+(2\lambda-\lambda^2)U^\ast U))\backslash\{0\}=\sigma(U)\backslash\{0\}.$$
Therefore, it suffices to prove that $U$ is invertible if and only if so is $P_{\lambda}(U)$. Indeed, it follows immediately from Theorem \ref{invertiblepmt} and Remark \ref{invpmt1} (i) since $\mathcal{R}(U)$ is closed.
\end{proof}

The spectra of an operator $T\in\mathbb{B}(\mathcal{H})$ and its power mean transform $P_{\lambda}(T)$ are not necessarily related. For the bilateral weighted shift operator $W_{\alpha}$ in the example below, we even have $\sigma(W_{\alpha})\bigcap\sigma(P_{\lambda}(W_{\alpha}))=\varnothing$ for $\lambda\in(0,1)$. Recall from \cite[Proposition 2.11]{ani-rrac-2023} that the spectrum of $W_{\alpha}$ is given by 
\begin{equation}\label{spws}
\sigma(W_{\alpha})=\{\gamma\in\mathbb{C}:|\gamma|\le r(W_{\alpha})\}
\end{equation}
if $W_{\alpha}$ is not invertible, and  \begin{equation}\label{spws1}
\sigma(W_{\alpha})=\{\gamma\in\mathbb{C}:\frac{1}{r(W_{\alpha}^{-1})}\le|\gamma|\le r(W_{\alpha})\}
\end{equation}
if $W_{\alpha}$ is invertible.
\begin{example}\label{estisp}
Let $W_{\alpha}$ be the bilateral weighted shift, where $\alpha_{n}=\begin{cases}
	1&\text{$n$ is even}\\
	4&\text{$n$ is odd}
\end{cases}$. It follows that $((1-\lambda)\sqrt{\alpha_{n}}+\lambda\sqrt{\alpha_{n+1}})^2=\begin{cases}
	(1+\lambda)^2&\text{$n$ is even}\\(2-\lambda)^2&\text{$n$ is odd}
\end{cases}.$ Since $\inf\limits_{n\in\mathbb{Z}}\alpha_{n}=1$ and $\inf\limits_{n\in\mathbb{Z}}((1-\lambda)\sqrt{\alpha_{n}}+\lambda\sqrt{\alpha_{n+1}})^2=\min\{(1+\lambda)^2,(2-\lambda)^2\}>0$, by \eqref{meanuni} we obtain that $W_{\alpha}$ and $P_{\lambda}(W_{\alpha})$ are invertible. By \cite[Remark 2.8 (ii)]{ani-rrac-2023} and Gelfand's spectral radius formula, together with \eqref{spws1}, we have $\sigma(W_{\alpha})=\{\gamma:|\gamma|=2\}$ and $\sigma(P_{\lambda}(W_{\alpha}))=\{\gamma:|\gamma|=(1+\lambda)(2-\lambda)\}$. Clearly, $\sigma(W_{\alpha})\bigcap\sigma(P_{\lambda}(W_{\alpha}))=\varnothing$ for $\lambda\in(0,1)$.
\end{example}
But for a class of bilateral weighted shift operators, we have the following spectral inclusion.
\begin{proposition}\label{spin}
Let $\lambda\in[0,1]$ and $W_{\alpha}$ be a bilateral weighted shift such that $P_{\lambda}(W_{\alpha})$ is non-invertible. Then
\begin{equation}\label{spin1}
4\lambda(1-\lambda)\sigma(W_{\alpha})\equiv\{4\lambda(1-\lambda)\gamma:\gamma\in\sigma(W_{\alpha})\}\subseteq\sigma(P_{\lambda}(W_{\alpha})).
\end{equation}
\end{proposition}
\begin{proof}
Note that $r(T)=r(\Delta(T))$ for $T\in\mathbb{B}(\mathcal{H})$ (see \cite[Theorem 4.12]{zhou-glma-2023} ). By the arithmetic-geometric mean inequality, we have	
\begin{equation}\label{arge}
	4\lambda(1-\lambda)\sqrt{\alpha_{i}\alpha_{i+1}}=\Big(2{\sqrt{\lambda(1-\lambda)\sqrt{\alpha_{i}\alpha_{i+1}}}}\Big)^2\le\Big((1-\lambda)\sqrt{\alpha_{i}}+\lambda\sqrt{\alpha_{i+1}}\Big)^2
\end{equation}
for $i\in\mathbb{Z}$. Then by \eqref{meanuni} we obtain
\begin{align*}
	4\lambda(1-\lambda)r(W_{\alpha})=4\lambda(1-\lambda)r(\Delta(W_{\alpha}))&=4\lambda(1-\lambda)\lim\limits_{n\to\infty}\Bigg(\sup\limits_{k\in\mathbb{Z}}\prod_{i=k}^{k+n-1}\sqrt{\alpha_{i}\alpha_{i+1}}\Bigg)^{\frac{1}{n}}\\&\le\lim\limits_{n\to\infty}\Bigg(\sup\limits_{k\in\mathbb{Z}}\prod\limits_{i=k}^{k+n-1}\Big((1-\lambda)\sqrt{\alpha_{i}}+\lambda\sqrt{\alpha_{i+1}}\Big)^2\Bigg)^{\frac{1}{n}}\\&=r(P_{\lambda}(W_{\alpha})).
\end{align*}
From Proposition \ref{inws}, we know that $W_{\alpha}$ is not invertible. Consequently, the desired inclusion follows from \eqref{spws} and the above inequality.
\end{proof}
\begin{remark}\label{empty}
The condition that $P_{\lambda}(W_{\alpha})$ is non-invertible cannot be removed. Consider the bilateral weighted shift operator $W_{\alpha}$ in Example \ref{estisp}. If $x\in4\lambda(1-\lambda)\sigma(W_{\alpha})\bigcap\sigma(P_{\lambda}(W_{\alpha}))$, then the modulus of $x$ satisfies $|x|=4\lambda(1-\lambda)\cdot2=(1+\lambda)(2-\lambda)$. However, the quadratic equation in $\lambda$ has no solution in $[0,1]$. Thus $4\lambda(1-\lambda)\sigma(W_{\alpha})\bigcap\sigma(P_{\lambda}(W_{\alpha}))=\varnothing$.
\end{remark}

Under the assumption of Proposition \ref{spin}, for $\lambda\in[0,1]$ the inequality $$4\lambda(1-\lambda)r(W_{\alpha})\le\|P_{\lambda}(W_{\alpha})\|$$ follows from \eqref{spin1} and \eqref{spnuno}. For any $T\in\mathbb{B}(\mathcal{H})$, the following theorem provides a norm estimate for $P_{\lambda}(T)$ and shows that the above inequality also holds for $T$. Before that, we need the lemma below (see also \cite[Proposition 1 (8)]{seddik-mia-2025}), which is a variant of the Heinz inequality. For the reader’s convenience, we give a proof.
\begin{lemma}\label{heinzvar}
Let $X\in\mathbb{B}(H)$. Suppose $A$ and $B$ are positive. Then
\begin{equation}\label{normesti}
	\Big\|2A^\frac{1}{2}XB^\frac{1}{2}+AX+XB\big\|\ge4\Big\|A^\frac{1}{2}XB^\frac{1}{2}\Big\|.
\end{equation}	
\end{lemma}
\begin{proof}
By continuity, it is enough to prove Lemma \ref{heinzvar} for the invertible operators $A+\varepsilon I$ and $B+\varepsilon I$, where $\varepsilon>0$ and $I$ is the identity operator on $\mathcal{H}$. Hence we can assume that $A$ and $B$ are invertible. The Heinz inequality states that if $P$ and $Q$ are positive invertible operators and $W$ is any operator, then\begin{equation}
	2\|W\|\le\Big\|P^{\frac{1}{2}}WQ^{-\frac{1}{2}}+P^{-\frac{1}{2}}WQ^{\frac{1}{2}}\Big\|.
\end{equation} 
Let $Z=A^\frac{1}{2}XB^\frac{1}{2}$ and $Y=A^\frac{1}{4}ZB^{-\frac{1}{4}}+A^{-\frac{1}{4}}ZB^\frac{1}{4}$. Then, applying Heinz's inequality twice, we obtain
\begin{align*}
	\Big\|AX+XB+2A^\frac{1}{2}XB^\frac{1}{2}\Big\|&=\Big\|A^\frac{1}{4}YB^{-\frac{1}{4}}+A^{-\frac{1}{4}}YB^\frac{1}{4}\Big\|\\&\ge2\Big\|A^\frac{1}{4}ZB^{-\frac{1}{4}}+A^{-\frac{1}{4}}ZB^{\frac{1}{4}}\Big\|\\&\ge4\Big\|A^\frac{1}{2}XB^\frac{1}{2}\Big\|,
\end{align*}
as desired.
\end{proof}

\begin{theorem}\label{spectral}
	Let $\lambda\in[0,1]$ and $T\in\mathbb{B}(\mathcal{H})$. Then
	\begin{equation}\label{thm2}
		4\lambda(1-\lambda)\|\Delta(T)\|\le\|P_{\lambda}(T)\|\le(1-2\lambda+2\lambda^2)\Big\|M_{\frac{(1-\lambda)^2}{(1-\lambda)^2+\lambda^2}}(T)\Big\|+(2\lambda-2\lambda^2)\|\Delta(T)\|.
	\end{equation}
	In particular, 
	\begin{equation}\label{normest2}
		4\lambda(1-\lambda)r(T)\le\|P_{\lambda}(T)\|\le\|T\|.
	\end{equation}
\end{theorem}
\begin{proof}
From the definition of $P_{\lambda}(T)$, we can also write\begin{equation}\label{exp}
	P_{\lambda}(T)=(1-2\lambda+2\lambda^2)M_{\frac{(1-\lambda)^2}{(1-\lambda)^2+\lambda^2}}(T)+(2\lambda-2\lambda^2)\Delta(T)
\end{equation} Let $T=U|T|$ be the polar decomposition of $T$. Substituting $A=4\lambda^2|T|$, $B=4(1-\lambda)^2|T|$ and $X=U$ into \eqref{normesti}, we obtain$$\Big\|8\lambda(1-\lambda)\Delta(T)+4\lambda^2 T^D+4(1-\lambda)^2T\Big\|\ge16\lambda(1-\lambda)\|\Delta(T)\|.$$
From the above inequality and \eqref{exp}, we immediately obtain \eqref{thm2}. Furthermore, by the submultiplicativity of the norm, we have
$\|T^D\|\le\|T\|$ and $\|\Delta(T)\|\le\|T\|$. Then combining the definition of $M_{\frac{(1-\lambda)^2}{(1-\lambda)^2+\lambda^2}}(T)$, \eqref{thm2} and $r(T)=r(\Delta(T))\le\|\Delta(T)\|$, we directly get \eqref{normest2}. The proof is complete.
\end{proof}
\begin{remark}
It should be mentioned that\begin{equation}\label{maxnorm}
	(2\lambda-2\lambda^2)\|\Delta(T)\|\le(1-2\lambda+2\lambda^2)\Big\|M_{\frac{(1-\lambda)^2}{(1-\lambda)^2+\lambda^2}}(T)\Big\|
\end{equation}
for $\lambda\in[0,1]$. This follows from $2\sqrt{t-t^2}\|\Delta(T)\|\le\|M_{t}(T)\|$ for $t\in[0,1]$ (see \cite[Theorem 3.2]{zamani-jmaa-2021}) and $t=\frac{(1-\lambda)^2}{(1-\lambda)^2+\lambda^2}$.
\end{remark}

If $T$ is quasinormal, then $P_{\lambda}(T)=M_{\frac{(1-\lambda)^2}{(1-\lambda)^2+\lambda^2}}(T)=\Delta(T)$, which yields the right‑hand equality in \eqref{thm2}. To prove that the equality on the left‑hand side of \eqref{thm2} can hold for $\lambda\in(0,1)$, it suffices to find an operator $T$ such that $4\lambda(1-\lambda)\Delta(T)=P_{\lambda}(T)$. Such an example is given below.
\begin{example}
Let $\mathscr{X}=\{1,2,3\}$ and let $\mu$ be the measure on the power set $2^\mathscr{X}$ defined by $\mu(\{x\})=\begin{cases}
	1&x=1,2\\0&x=3
\end{cases}.$ Take $\lambda\in(0,1)$. Define the transform $\phi$ on $\mathscr{X}$ by $\phi(x)=\begin{cases}
	2&x=1\\
	3&x=2\\
	1&x=3
\end{cases}$ and the function $\omega:\mathscr{X}\to\mathbb{R}$ by $\omega(x)=\begin{cases}
	1&x=1\\
	(\frac{\lambda}{1-\lambda})^2&x=2\\
	(\frac{1-\lambda}{\lambda})^2&x=3
\end{cases}.$ Consider the weighted composition operator $C_{\phi,\omega}$ on the Hilbert space $L^2(\mu)$ of square‑integrable functions with respect to $\mu$, defined by $$C_{\phi,\omega}f=\omega\cdot f\circ\phi, \quad f\in L^2(\mu).$$By \cite[Theorem 18 and Proposition 79]{b-j-j-sW}, the polar decomposition of $C_{\phi,\omega}$ is $C_{\phi,\omega}=U|C_{\phi,\omega}|$ with $Uf=\frac{\omega}{(\mathsf{h}_{\phi,\omega}\circ\phi)^{\frac{1}{2}}}\cdot f\circ\phi$ and $|C_{\phi,\omega}|f=\mathsf{h}_{\phi,\omega}^{1/2}\cdot f$, where $$\mathsf{h}_{\phi,\omega}(x)=\begin{cases}
	(\frac{1-\lambda}{\lambda})^4&x=1\\
	1&x=2\\
	(\frac{\lambda}{1-\lambda})^4&x=3
\end{cases}\quad\text{and}\quad\mathsf{h}_{\phi,\omega}\circ\phi(x)=\begin{cases}
	1&x=1\\
	(\frac{\lambda}{1-\lambda})^4&x=2\\
	(\frac{1-\lambda}{\lambda})^4&x=3
\end{cases}.$$  Then 
\begin{equation}\label{comdelta}
	\Delta(C_{\phi,\omega})f=(\frac{\mathsf{h}_{\phi,\omega}}{\mathsf{h}_{\phi,\omega}\circ\phi})^\frac{1}{4}\cdot\omega\cdot f\circ\phi,
\end{equation} and from the definition of $P_{\lambda}(C_{\phi,\omega})$ it is not difficult to obtain that
\begin{equation}\label{compower}
	P_{\lambda}(C_{\phi,\omega})f=((1-\lambda)+\lambda(\frac{\mathsf{h}_{\phi,\omega}}{\mathsf{h}_{\phi,\omega}\circ\phi})^\frac{1}{4})^2\cdot \omega\cdot f\circ\phi
\end{equation}  (see also Theorem 3.2 in \cite{benhida-mn-2020}, where replacing $\alpha$ by $1$ and $\frac{1}{2}$ gives $C_{\phi,\omega}^D$ and $\Delta(C_{\phi,\omega})$, respectively). A direct computation shows that for $x\in\{1,2\}$, $$4\lambda(1-\lambda)(\frac{\mathsf{h}_{\phi,\omega}}{\mathsf{h}_{\phi,\omega}\circ\phi})^\frac{1}{4}(x)=((1-\lambda)+\lambda(\frac{\mathsf{h}_{\phi,\omega}}{\mathsf{h}_{\phi,\omega}\circ\phi}))^\frac{1}{4})^2(x)=4(1-\lambda)^2$$
Since $\mu(\{3\})=0$, combining the above expression with \eqref{comdelta} and \eqref{compower} yields $$4\lambda(1-\lambda)\Delta(C_{\phi,\omega})=P_{\lambda}(C_{\phi,\omega}),$$ as desired.
\end{example}

As a consequence of Theorem \ref{spectral}, we immediately get a spectral radius equation via the power mean transform.

\begin{corollary}
	Let $\lambda\in(0,1)$ and $T\in\mathbb{B}(\mathcal{H})$. Then $r(T)=\lim\limits_{n\to\infty}\|P_{\lambda}(T^n)\|^\frac{1}{n}$.
\end{corollary}
\begin{proof}
	Let $n\in\mathbb{N}$. From \eqref{normest2} and $r(T^n)=(r(T))^n$ we obtain $$4\lambda(1-\lambda)(r(T))^n=4\lambda(1-\lambda)r(T^n)\le\|P_{\lambda}(T^n)\|\le\|T^n\|.$$ The desired result then follows from $r(T)=\lim\limits_{n\to\infty}\|T^n\|^{\frac{1}{n}}$.
\end{proof}
The next result shows that on a subset of $\mathbb{B}(\mathcal{H})$, the map $P_{\lambda}:T\mapsto P_{\lambda}(T)$ is not a scalar multiplication in most cases.
\begin{corollary}
	Let $\lambda\in(0,1)$ and $T\in\mathbb{B}(\mathcal{H})$. Then the following conditions are equivalent$\colon$
	\begin{itemize}
		\item[(i)] $P_{\lambda}(T)=cT$ for all $c\in (-\infty,(1-\lambda)^2)\bigcup(1,+\infty)$.
		\item[(ii)] $P_{\lambda}(T)=cT$ for some $c\in (-\infty,(1-\lambda)^2)\bigcup(1,+\infty)$.
			\item[(iii)] $P_{\lambda}(T)=f(\lambda)\Delta(T)$ for some function $f$ on $[0,1]$ satisfying $|f(\lambda)|<4\lambda(1-\lambda)$
		\item[(iv)] $T=0$.
	\end{itemize}
\end{corollary}
\begin{proof}
	It is clear that (i) implies (ii). Moreover, (iv) immediately implies the others by the definition of $P_{\lambda}(T)$. We finish the proof by proving (iii) $\Rightarrow$ (ii) $\Rightarrow$ (iv).
	
	(iii) $\Rightarrow$ (ii): From \eqref{thm2} we have $|f(\lambda)|\|\Delta(T)\|\ge4\lambda(1-\lambda)\|\Delta(T)\|$, which forces $\|\Delta(T)\|=0$ and hence $\Delta(T)=0$. Then $P_{\lambda}(T)=f(\lambda)\Delta(T)=0=0\cdot\Delta(T)$, so (ii) follows immediately.
	
	(ii) $\Rightarrow$ (iv): The proof is divided into two cases depending on the value of $c$.
	
	Case 1: $c\in(1,+\infty)$. From \eqref{normest2} we have
	$$c\|T\|=\|cT\|=\|P_{\lambda}(T)\|\le\|T\|.$$
	The condition $c>1$ together with the above inequality directly implies $T=0$.
	
	Case 2: $c\in(-\infty,(1-\lambda)^2)$. Let $T=U|T|$ be the polar decomposition of $T$. Suppose $P_{\lambda}(T)=cT$, then we have $U^\ast (P_{\lambda}(T)-cT)=0$, which implies, by using the fact $U^\ast U=P_{\overline{\mathcal{R}(|T|)}}$, that\begin{equation}\label{zero1}
		((1-\lambda)^2-c)|T|+\lambda^2U^\ast|T|U+2\lambda(1-\lambda)U^\ast|T|^\frac{1}{2}U|T|^\frac{1}{2}=0.
	\end{equation}
	By taking the adjoint, we get\begin{equation}\label{zero2}
		((1-\lambda)^2-c)|T|+\lambda^2U^\ast|T|U+2\lambda(1-\lambda)|T|^\frac{1}{2}U^\ast|T|^\frac{1}{2}U=0.
	\end{equation}
	It follows from equations \eqref{zero1} and \eqref{zero2} that the positive operators $U^\ast|T|^\frac{1}{2}U$ and $|T|^\frac{1}{2}$ commute, which yields that $$U^\ast|T|^\frac{1}{2}U|T|^\frac{1}{2}\ge0.$$ On the other hand, using \eqref{zero1} together with the fact that $((1-\lambda)^2-c)|T|$ and $\lambda^2U^\ast|T|U$ are positive, we see that $$-U^\ast|T|^\frac{1}{2}U|T|^\frac{1}{2}=\frac{1}{2\lambda(1-\lambda)}\Big(((1-\lambda)^2-c)|T|+\lambda^2U^\ast|T|U\Big)\ge0.$$
	Then we deduce $$U^\ast|T|^\frac{1}{2}U|T|^\frac{1}{2}=0.$$ Substituting the above equation into \eqref{zero1} gives $((1-\lambda)^2-c)|T|+\lambda^2U^\ast|T|U=0$. We conclude that $|T|=0$ from the fact that $((1-\lambda)^2-c)|T|$ and $\lambda^2U^\ast|T|U$ are positive operators. Hence, $T=U|T|=0$. 
	
	The proof is complete.
\end{proof}
\begin{remark}
	(i) If $\lambda=0$, then $P_{\lambda}(T)=P_{0}(T)=T$; hence this case is trivial. If $\lambda=1$, then (ii) $\Rightarrow$ (iv) does not hold in general. Indeed, from Remark \ref{duggal} we have $P_{1}(T)=T^D=0=0\cdot T$, but $T\ne0$.
	
	(ii) The restriction $|f(\lambda)|<4\lambda(1-\lambda)$ in condition (iii) cannot be omitted. For example, take $f(\lambda)=1$ and $\lambda=\frac{1}{2}$; then clearly there exists a nonzero quasinormal operator satisfying $P_{\frac{1}{2}}(T)=\Delta(T)$.
	
	(iii) If $c\in[(1-\lambda)^2,1]$, then condition (ii) does not generally imply (iv). For the case $c=(1-\lambda)^2$, consider the matrix $T$ defined in Remark \ref{duggal}. A direct computation gives $$P_{\lambda}(T)=(1-\lambda)^2T+\lambda^2T^D+2\lambda(1-\lambda)\Delta(T)=(1-\lambda)^2T,$$but $T\ne0$. For the case $c\in((1-\lambda)^2,1]$. Let $\alpha\equiv\{\alpha_{n}\}_{n\in\mathbb{N}}$ be a sequence such that $\alpha_{n+1}=\Big(\frac{\sqrt{c}-(1-\lambda)}{\lambda}\Big)^2\alpha_{n}$ and $\alpha_{0}=1$. Note that the sequence is bounded since the common ratio $\Big(\frac{\sqrt{c}-(1-\lambda)}{\lambda}\Big)^2\le1$. Then we consider the unilateral weighted shift $W_{\alpha}$ induced by the sequence $\{\alpha_{n}\}$. From \eqref{uni} we have
	\begin{align*}
		P_{\lambda}(W_{\alpha})e_{n}&=((1-\lambda)\sqrt{\alpha_{n}}+\lambda\sqrt{\alpha_{n+1}})^2e_{n+1}\\&=((1-\lambda)\sqrt{\alpha_{n}}+\lambda\frac{\sqrt{c}-(1-\lambda)}{\lambda}\sqrt{\alpha_{n}})^2e_{n+1}\\&=c\alpha_{n}e_{n+1}=cW_{\alpha}e_{n}
	\end{align*}
	(that is, $P_{\lambda}(W_{\alpha})=cW_{\alpha}$). However, $W_{\alpha}\ne0$.
\end{remark}

Let $x,y\in\mathcal{H}$. Then the at most rank one operator $x\otimes y$ is defined by $(x\otimes y)z=\langle z,y\rangle z$ for $z\in\mathcal{H}$. One can easily see
\begin{equation}\label{normrankone}
	\|x\otimes y\|=\|x\|\|y\|.
\end{equation}
The power mean transform of $x\otimes y$ yields an improvement of the Cauchy-Schwarz inequality.
\begin{corollary}
Let $x,y\in\mathcal{H}$ with $y\ne0$. Then 
\begin{align*}
|\langle x,y\rangle|\le\min\limits_{t\in[0,1]}\|tx+(1-t)\frac{\langle x,y\rangle}{\|y\|^2}y\|\|y\|&\le\min\limits_{t\in[0,1]}\Big(\|(1-t)^2x+t^2\frac{\langle x,y\rangle}{\|y\|^2}y\|\|y\|+2t(1-t)|\langle x,y\rangle|\Big)\\&\le\|x\|\|y\|
\end{align*}
\end{corollary}
\begin{proof}
Let $\lambda\in[0,1]$. From \cite[Proposition 2.1]{chabbabi-jmaa-2017} and \cite[Theorem 4.19]{zamani-jmaa-2021} we have $$\Delta(x\otimes y)=\Big(\frac{\langle x,y\rangle}{\|y\|^2}y\Big)\otimes y\ \text{and}\ M_{\frac{(1-\lambda)^2}{(1-\lambda)^2+\lambda^2}}(x\otimes y)=\Big(\frac{(1-\lambda)^2}{(1-\lambda)^2+\lambda^2}x+\frac{\lambda^2}{(1-\lambda)^2+\lambda^2}\frac{\langle x,y\rangle}{\|y\|^2}y\Big)\otimes y.$$ Consequently from \eqref{exp} we obtain
\begin{equation}\label{rankpower}
P_{\lambda}(x\otimes y)=((1-\lambda)^2x+(1-(1-\lambda)^2)\frac{\langle x,y\rangle}{\|y\|^2}y)\otimes y.
\end{equation}
Together with \eqref{normrankone}, the above result yields $\|\Delta(x\otimes y)\|=|\langle x,y\rangle|$, $$(1-2\lambda+2\lambda^2)\|M_{\frac{(1-\lambda)^2}{(1-\lambda)^2+\lambda^2}}(x\otimes y)\|=\|(1-\lambda)^2x+\lambda^2\frac{\langle x,y\rangle}{\|y\|^2}y\|\|y\|$$and $$\|P_{\lambda}(x\otimes y)\|=\|(1-\lambda)^2x+(1-(1-\lambda)^2)\frac{\langle x,y\rangle}{\|y\|^2}y\|\|y\|.$$
After substituting the above expressions into Theorem \ref{spectral} and then using the identity $$\min\limits_{\lambda\in[0,1]}\|P_{\lambda}(x\otimes y)\|=\min\limits_{t\in[0,1]}\|tx+(1-t)\frac{\langle x,y\rangle}{\|y\|^2}y\|\|y\|,$$ we obtain the last two inequalities in this corollary. Now, if we can prove that 
\begin{equation}\label{coroineq}
	|\langle x,y\rangle|\le\|tx+(1-t)\frac{\langle x,y\rangle}{\|y\|^2}y\|\|y\|
\end{equation}for all $t\in[0,1]$, then the first inequality holds, and the proof is complete. Indeed, by orthogonal decomposition we can write $x=\alpha y+z$, where $\langle y,z\rangle=0$. Hence, it follows readily that \eqref{coroineq} is equivalent to $|\alpha|\|y\|\le\|\alpha y+tz\|$, and the latter is obviously true because $\langle\alpha y,tz\rangle=0$.
\end{proof}
\begin{remark}
 Let $\mathcal{H}=\mathbb{C}^2$, $x=(1,0)$ and $y=(1,1)$. Then a direct computation yields $|\langle x,y\rangle|=1$, $\|x\|\|y\|=\sqrt{2}$, $\min\limits_{t\in[0,1]}\|tx+(1-t)\frac{\langle x,y\rangle}{\|y\|^2}y\|\|y\|=\frac{\sqrt{2}}{2}\min\limits_{t\in[0,1]}\sqrt{2t^2+2}=1$
and
\begin{align*}
&\min\limits_{t\in[0,1]}\Big(\|(1-t)^2x+t^2\frac{\langle x,y\rangle}{\|y\|^2}y\|\|y\|+2t(1-t)|\langle x,y\rangle|\Big)\\&=\min\limits_{t\in[0,1]}\Big(\sqrt{(1-2t+2t^2)^2+(1-t)^4}+2t(1-t)\Big)=1.
\end{align*}
The last equality holds because $$\sqrt{(1-2t+2t^2)^2+(1-t)^4}+2t(1-t)\ge\sqrt{(1-2t+2t^2)^2}+2t(1-t)=1$$ and by taking $t=1$. This simple example illustrates the above corollary.
\end{remark}
Theorem \ref{spectral} together with the theorem and lemma below provides a characterization of $\|P_{\lambda}(T)\|=\|T\|$ using the numerical range.
\begin{theorem}\cite[Theorem 2.1]{B.B}\label{daugavet1} Let $A,B\in\mathbb{B}(\mathcal{H})$. Then $\|A+B\|=\|A\|+\|B\|$ holds if and only if $\|A\|\|B\|\in\overline{W(A^\ast B)}$.
\end{theorem}
\begin{lemma}\cite[Lemma 2.1]{Abr-JFA-1991}\label{daugavet2} In a normed space, if vectors $u$ and $v$ satisfy $\|u+v\|=\|u\|+\|v\|$, then the same property holds for any non-negative scalars $\alpha$ and $\beta$ in the equation $\|\alpha u+\beta v\|=\|\alpha u\|+\|\beta v\|$.
\end{lemma}
\begin{corollary}
Let $\lambda\in(0,1)$ and $T\in\mathbb{B}(\mathcal{H})$. Then the following conditions are equivalent$\colon$
\begin{itemize}
\item[(i)] $\|P_{\lambda}(T)\|=\|T\|$.
\item[(ii)] $\|T\|^2\in\overline{W\Big((M_{\frac{(1-\lambda)^2}{(1-\lambda)^2+\lambda^2}}(T))^\ast\Delta(T)\Big)}$.
\end{itemize}
\end{corollary}
\begin{proof}
(i) $\Rightarrow$ (ii): From Theorem \ref{spectral} we have $$\|P_{\lambda}(T)\|=(1-2\lambda+2\lambda^2)\Big\|M_{\frac{(1-\lambda)^2}{(1-\lambda)^2+\lambda^2}}(T)\Big\|+(2\lambda-2\lambda^2)\|\Delta(T)\|=\|T\|.$$
Then $\Big\|M_{\frac{(1-\lambda)^2}{(1-\lambda)^2+\lambda^2}}(T)\Big\|=\|\Delta(T)\|=\|T\|$ since $\Big\|M_{\frac{(1-\lambda)^2}{(1-\lambda)^2+\lambda^2}}(T)\Big\|\le\|T\|$ and $\|\Delta(T)\|\le\|T\|$. Moreover, $\Big\|M_{\frac{(1-\lambda)^2}{(1-\lambda)^2+\lambda^2}}(T)+\Delta(T)\Big\|=\Big\|M_{\frac{(1-\lambda)^2}{(1-\lambda)^2+\lambda^2}}(T)\Big\|+\|\Delta(T)\|$ follows from Lemma \ref{daugavet2}. Consequently, by Theorem \ref{daugavet1} we obtain$$\|T\|^2=\Big\|M_{\frac{(1-\lambda)^2}{(1-\lambda)^2+\lambda^2}}(T)\Big\|\|\Delta(T)\|\in\overline{W\Big((M_{\frac{(1-\lambda)^2}{(1-\lambda)^2+\lambda^2}}(T))^\ast\Delta(T)\Big)}.$$
(ii) $\Rightarrow$ (i): Since $\|T\|^2\in\overline{W\Big((M_{\frac{(1-\lambda)^2}{(1-\lambda)^2+\lambda^2}}(T))^\ast\Delta(T)\Big)}$, it follows from \eqref{spnuno} that
\begin{small}
\begin{equation*}
	\|T\|^2\le\omega\Big((M_{\frac{(1-\lambda)^2}{(1-\lambda)^2+\lambda^2}}(T))^\ast\Delta(T)\Big)\le\Big\|(M_{\frac{(1-\lambda)^2}{(1-\lambda)^2+\lambda^2}}(T))^\ast\Delta(T)\Big\|\le\Big\|M_{\frac{(1-\lambda)^2}{(1-\lambda)^2+\lambda^2}}(T)\Big\|\|\Delta(T)\|\le\|T\|^2.
\end{equation*}
\end{small}Hence,
\begin{equation}\label{meandeltaeq}
\Big\|M_{\frac{(1-\lambda)^2}{(1-\lambda)^2+\lambda^2}}(T)\Big\|=\|\Delta(T)\|=\|T\|
\end{equation}
 and $ \Big\|M_{\frac{(1-\lambda)^2}{(1-\lambda)^2+\lambda^2}}(T)\Big\|\|\Delta(T)\|=\|T\|^2\in\overline{W\Big((M_{\frac{(1-\lambda)^2}{(1-\lambda)^2+\lambda^2}}(T))^\ast\Delta(T)\Big)}$. The latter implies, by Theorem \ref{daugavet1}, that $\Big\|M_{\frac{(1-\lambda)^2
 	}{(1-\lambda)^2+\lambda^2}}(T)+\Delta(T)\Big\|=\Big\|M_{\frac{(1-\lambda)^2}{(1-\lambda)^2+\lambda^2}}(T)\Big\|+\|\Delta(T)\|$. The desired result follows from Lemma \ref{daugavet2} and \eqref{meandeltaeq}.
\end{proof}
\begin{remark}
If $\lambda=0$, then $P_{\lambda}(T)=P_{0}(T)=T$; hence this case is trivial. If $\lambda=1$, then the above corollary does not hold in general. Indeed, let $T=U|T|=\begin{bmatrix}
0&1\\1&0
\end{bmatrix}\begin{bmatrix}
1&0\\0&2
\end{bmatrix}=\begin{bmatrix}
	0&2\\1&0
\end{bmatrix}.$ Then $P_{1}(T)=M_{0}(T)=T^D=\begin{bmatrix}
0&1\\2&0
\end{bmatrix}$ and hence $\|T^D\|=\|T\|=2$. Set $x=(a,b)\in\mathbb{C}^2$ such that $|a|^2+|b|^2=1$. A direct computation gives $$\langle\Delta(T)x,T^Dx\rangle=\langle(\sqrt{2}b,a),(b,2a)\rangle=\sqrt{2}|b|^2+2\sqrt{2}|a|^2=\sqrt{2}+\sqrt{2}|a|^2\le2\sqrt{2}$$ since $\Delta(T)=\begin{bmatrix}
	0&\sqrt{2}\\\sqrt{2}&0
\end{bmatrix}$. Clearly, $4\notin\overline{W((T^D)^\ast\Delta(T))}$, as desired.
\end{remark}

As in the above argument, we also present the following.
\begin{corollary}
Let $\lambda\in(0,1)$ and $T\in\mathbb{B}(\mathcal{H})$. Then the following conditions are equivalent$\colon$
\begin{itemize}
	\item[(i)] $\|P_{\lambda}(T)\|=2(1-2\lambda+2\lambda^2)\Big\|M_{\frac{(1-\lambda)^2}{(1-\lambda)^2+\lambda^2}}(T)\Big\|$.
	\item[(ii)] $\frac{1-2\lambda+2\lambda^2}{2\lambda-2\lambda^2}\Big\|M_{\frac{(1-\lambda)^2}{(1-\lambda)^2+\lambda^2}}(T)\Big\|^2\in\overline{W\Big((M_{\frac{(1-\lambda)^2}{(1-\lambda)^2+\lambda^2}}(T))^\ast\Delta(T)\Big)}$.
\end{itemize}
In particular, $\|P(T)\|=\|M(T)\|$ if and only if  $\|M(T)\|^2\in\overline{W((M(T))^\ast\Delta(T))}$.
\end{corollary}
\begin{proof}
	For convenience, we write $R=(1-2\lambda+2\lambda^2)M_{\frac{(1-\lambda)^2}{(1-\lambda)^2+\lambda^2}}(T)$ and $S=(2\lambda-2\lambda^2)\Delta(T)$. From \eqref{maxnorm}, we have $\|R\|\ge\|S\|$.
	
(i) $\Rightarrow$ (ii): From \eqref{thm2} we get $\|P_{\lambda}(T)\|\le\|R\|+\|S\|\le2\|R\|$; hence, by the assumption, $\|P_{\lambda}(T)\|=\|R\|+\|S\|$ and $\|R\|=\|S\|$. Then, using Theorem \ref{daugavet1} and Lemma \ref{daugavet2} we obtain 
\begin{equation}\label{normequality}
\Big\|M_{\frac{(1-\lambda)^2}{(1-\lambda)^2+\lambda^2}}(T)\Big\|\|\Delta(T)\|\in\overline{W\Big((M_{\frac{(1-\lambda)^2}{(1-\lambda)^2+\lambda^2}}(T))^\ast\Delta(T)\Big)}.
\end{equation}
Thus, (ii) follows from \eqref{spnuno} and $\Big\|M_{\frac{(1-\lambda)^2}{(1-\lambda)^2+\lambda^2}}(T)\Big\|\|\Delta(T)\|=\frac{\|R\|\|S\|}{(1-2\lambda+2\lambda^2)(2\lambda-2\lambda^2)}=\frac{\|R\|^2}{(1-2\lambda+2\lambda^2)(2\lambda-2\lambda^2)}$.

(ii) $\Rightarrow$ (i): 

From the assumption and the arithmetic-geometric mean inequality we obtain
\begin{align*}
\frac{\|R\|^2}{(1-2\lambda+2\lambda^2)(2\lambda-2\lambda^2)}&\le\omega\Big((M_{\frac{(1-\lambda)^2}{(1-\lambda)^2+\lambda^2}}(T))^\ast\Delta(T)\Big)\\&\le\Big\|M_{\frac{(1-\lambda)^2}{(1-\lambda)^2+\lambda^2}}(T)\Big\|\|\Delta(T)\|=\frac{\|R\|\|S\|}{(1-2\lambda+2\lambda^2)(2\lambda-2\lambda^2)}\\&\le\frac{\|R\|^2+\|S\|^2}{2(1-2\lambda+2\lambda^2)(2\lambda-2\lambda^2)}\le\frac{2\|R\|^2}{2(1-2\lambda+2\lambda^2)(2\lambda-2\lambda^2)},
\end{align*}
which yields \eqref{normequality} and $\|R\|=\|S\|$. Then, using Theorem \ref{daugavet1} and Lemma \ref{daugavet2} we have $$\|P_{\lambda}(T)\|=\|R\|+\|S\|=2\|R\|,$$
as desired.

The `in particular' part follows from $\lambda=\frac{1}{2}$.
\end{proof}
\begin{remark}
If $\lambda=1$, then (i) is easily seen to be equivalent to $T^D=0$. Consequently, under this hypothesis we have $\Delta(T)=0$, and therefore $\overline{W((M_{0}(T))^\ast\Delta(T))}=\{0\}$. Hence this case is trivial.
\end{remark}

At the end of this section, we discuss two numerical radius inequalities involving $P_{\lambda}(T)$. 

\begin{theorem}
Let $\lambda\in[0,1]$ and $T\in\mathbb{B}(\mathcal{H})$. Then
\begin{align*}
	\omega(P_{\lambda}(T))&\le2\int_{0}^1\omega (t(1-2\lambda+2\lambda^2)M_{\frac{(1-\lambda)^2}{(1-\lambda)^2+\lambda^2}}(T)+(1-t)(2\lambda-2\lambda^2)\Delta(T))dt\\&\le
	\frac{1}{2}\Big((1-2\lambda+2\lambda^2)\omega(M_{\frac{(1-\lambda)^2}{(1-\lambda)^2+\lambda^2}}(T))+(2\lambda-2\lambda^2)\omega(\Delta(T))+\omega(P_{\lambda}(T))\Big).
\end{align*}
\end{theorem}
\begin{proof}
	
Since $\omega(\cdot)$ is a norm on $\mathbb{B}(\mathcal{H})$, it is easy to verify that for any $A,B\in\mathbb{B}(\mathcal{H})$, the function $f(t)=\omega(tA+(1-t)B)$ is convex on $\mathbb{R}$. Then, by the Hammer-Bullen inequality \cite{Ni.Pe}, we have $$2f\Big(\frac{0+1}{2}\Big)\le\frac{2}{1-0}\int_{0}^1f(t)dt\le\frac{f(0)+f(1)}{2}+f\Big(\frac{0+1}{2}\Big),$$
i.e.,$$2\omega\Big(\frac{A+B}{2}\Big)\le2\int_{0}^1\omega(tA+(1-t)B) dt\le\frac{\omega(A)+\omega(B)}{2}+\omega\Big(\frac{A+B}{2}\Big).$$
Substituting $A=(1-2\lambda+2\lambda^2)M_{\frac{(1-\lambda)^2}{(1-\lambda)^2+\lambda^2}}(T)$ and $B=(2\lambda-2\lambda^2)\Delta(T)$ into the above inequality, we obtain the desired result.
\end{proof}
\begin{remark}
Feki and Yamazaki proved in \cite[p. 412 and Theorem 2]{feki-mia-2021} that
\begin{equation}\label{feki2}
\omega(T^D)\le\omega(T)
\end{equation}
and
\begin{equation}\label{feki3}
\omega(\Delta(T))\le\frac{1}{2}\omega(T)+\frac{1}{2}\omega(T^D)
\end{equation}
for $T\in\mathbb{B}(\mathcal{H})$. From the definition of $P_{\lambda}(T)$, \eqref{feki2} and \eqref{feki3}, together with the fact that $\omega(\cdot)$ is a norm, we immediately obtain
\begin{equation*}
	\omega(P_{\lambda}(T))\le\omega(T).
\end{equation*}
In fact, this theorem improves the above inequality. For example, Let $\lambda=\frac{1}{2}$ and $T=\begin{bmatrix}
0&1\\0&1
\end{bmatrix}=U|T|$ with $U=\begin{bmatrix}
0&\frac{\sqrt{2}}{2}\\0&\frac{\sqrt{2}}{2}
\end{bmatrix}$ and $|T|=\begin{bmatrix}
0&0\\0&\sqrt{2}
\end{bmatrix}$. Then $\Delta(T)=T^D=\begin{bmatrix}
0&0\\0&1
\end{bmatrix}$, $M(T)=\begin{bmatrix}
0&\frac{1}{2}\\0&1
\end{bmatrix}$ and $P(T)=\begin{bmatrix}
0&\frac{1}{4}\\0&1
\end{bmatrix}$. For $a\ge0$, we have $\omega\Big(\begin{bmatrix}
0&a\\0&1
\end{bmatrix}\Big)=\frac{1+\sqrt{a^2+1}}{2}$. Hence, $\omega(P(T))=\frac{4+\sqrt{17}}{8}$, $\omega(M(T))=\frac{2+\sqrt{5}}{4}$, $\omega(\Delta(T))=1$, $\omega(T)=\frac{1+\sqrt{2}}{2}$ and $$\displaystyle \int_{0}^{1}\omega\Big(\begin{bmatrix}
0&\frac{1}{2}t\\0&1
\end{bmatrix}\Big)dt=\frac{1}{4}\int_{0}^12+\sqrt{t^2+4}dt=\frac{4+\sqrt{5}}{8}+\frac{1}{2}\ln\frac{1+\sqrt{5}}{2}.$$
Therefore, we obtain
\begin{align*}
\omega(P(T))\approx 1.0154&\le\int_{0}^1\omega(tM(T)+(1-t)\Delta(T))dt=\int_{0}^1\omega\Big(\begin{bmatrix}
	0&\frac{1}{2}t\\0&1
\end{bmatrix}\Big)dt\approx 1.0201\\&\le\frac{1}{2}(\frac{1}{2}\omega(M(T))+\frac{1}{2}\omega(\Delta(T))+\omega(P(T)))\approx1.0224\\&\le\omega(T)\approx 1.207.
\end{align*}
\end{remark}
Recently, Stanković proved in \cite[Theorem 4.7 and Corollary 4.8]{Stankovic-ZAA-2024} that
\begin{equation}\label{stan1}
\omega(\Delta(T))\le\omega(H_{\lambda}(T))\le\frac{1}{2}\omega(\Delta_{\lambda}(T))+\frac{1}{2}\omega(\Delta_{1-\lambda}(T))
\end{equation}
for $\lambda\in[0,1]$ and $T\in\mathbb{B}(\mathcal{H})$, which generalizes \eqref{feki3}. For the matrix case, the inequality \eqref{feki3} can be extended using the power mean transform together with a recent result of Hosseini et al. Let $\mathbb{M}_{n}(\mathbb{C})$ be the set of $n\times n$ complex matrices.
\begin{theorem}\cite[Theorem 2.2]{moradi-jmi-2025}\label{hoss}
Let $A,X\in \mathbb{M}_{n}(\mathbb{C})$ such that $A$ be positive definite. Then for any $0\le v\le1$,
$$\omega(A^vXA^{1-v}+A^{1-v}XA^v)\le\omega(4rA^\frac{1}{2}XA^\frac{1}{2}+(1-2r)(AX+XA))$$
where $r=\min\{v,|\frac{1}{2}-v|,1-v\}$.
\end{theorem}
\begin{theorem}
Let $\lambda\in[0,1]$ and $T\in \mathbb{M}_{n}(\mathbb{C})$. Then
$$\omega(\Delta(T))\le\frac{1}{2}\omega(P_{\lambda}(T)+P_{1-\lambda}(T))\le\frac{1}{2}\omega(P_{\lambda}(T))+\frac{1}{2}\omega(P_{1-\lambda}(T)).$$
In particular, setting $\lambda=0$ or $\lambda=1$ gives \eqref{feki3}.
\end{theorem}
\begin{proof}
Let $T=U|T|$ be the polar decomposition of $T$. Note that $$\lambda(1-\lambda)=\min\Big\{\lambda(1-\lambda),\Big|\frac{1}{2}-\lambda(1-\lambda)\Big|,1-\lambda(1-\lambda)\Big\}.$$ Substituting $A=|T|$, $X=U$ and $r=v=\lambda(1-\lambda)$ into Theorem \ref{hoss}, we obtain
\begin{align*}
2\omega(\Delta(T))\le\omega(\Delta_{\lambda(1-\lambda)}(T)+\Delta_{1-\lambda(1-\lambda)}(T))&\le\omega(4\lambda(1-\lambda)\Delta(T)+((1-\lambda)^2+\lambda^2)(T+T^D))\\&=\omega(P_{\lambda}(T)+P_{1-\lambda}(T))\\&\le\omega(P_{\lambda}(T))+\omega(P_{1-\lambda}(T)),
\end{align*}
where the first inequality holds by \eqref{stan1}, and the last one holds because $\omega(\cdot)$ is a norm on $\mathbb{M}_{n}(\mathbb{C})$.
\end{proof}
	
	\section{The Power mean transform of binormal operators}
	
	We begin this part with a question posed by Golla et al. In \cite[Question 6]{yamazaki-laa-2023}, it is asked whether the equality $\Delta_{\mathsf{m}_{f}}(T)=T$ implies that $T$ is quasinormal. The result below shows that if $T$ is invertible and $f(x)=(\lambda+(1-\lambda)\sqrt{x})^2$ for $\lambda\in(0,1]$, then the answer is positive, since $\Delta_{\mathsf{m}_{f}}(T)=P_{\lambda}(T)$. Let us recall the following lemma due to Rosenblum.
	\begin{lemma}\cite{Djo-laa-2021, Ros-Duke-1956}\label{dukeequ}
		Let $A\in\mathbb{B}(\mathcal{K})$, $B\in\mathbb{B}(\mathcal{H})$ and $S\in\mathbb{B}(\mathcal{H},\mathcal{K})$. Then the equation $AX-XB=S$ admits a unique solution $X\in\mathbb{B}(\mathcal{H},\mathcal{K})$ if $\sigma(A)\bigcap\sigma(B)=\emptyset$.
	\end{lemma}
	
	\begin{theorem}
		Let $\lambda\in(0,1]$ and $T\in\mathbb{B}(\mathcal{H})$. Then $P_{\lambda}(T)=T$ if and only if $T$ is quasinormal.
	\end{theorem}
	\begin{proof}
		\textit{Necessity}. If $T$ is quasinormal, then $T=T^D=\Delta(T)$. The desired result follows immediately from the definition of $P_{\lambda}(T)$.
		
		\textit{Sufficiency}. Let $T=U|T|$ be the polar decomposition of $T$. Write $C=|T|^\frac{1}{2}U-U|T|^\frac{1}{2}$. If we can prove $C=0$, then $T=U|T|=U|T|^\frac{1}{2}|T|^\frac{1}{2}=|T|^\frac{1}{2}U|T|^\frac{1}{2}=|T|^\frac{1}{2}|T|^\frac{1}{2}U=T^D,$ which implies that $T$ is quasinormal. The case $T=0$ is trivial. We now consider $T\ne0$. In fact, $P_{\lambda}(T)=T$ yields
		\begin{equation}\label{syequ}
		\lambda^2|T|^\frac{1}{2}C+(1-(1-\lambda)^2)C|T|^\frac{1}{2}=0
		\end{equation}
		directly.
		Let $E$ be the spectral measure of $|T|$ on $[0,\|T\|]$. Since $T\ne0$, there exists a positive integer $N$ such that $\frac{1}{N}<\|T\|$. Set $\mathcal{P}_{0}=E(\{0\})=P_{\mathcal{N}(|T|)}$, $\mathcal{P}_{N}=E([\frac{1}{N},\|T\|])$ and $\mathcal{P}_{n+1}=E([\frac{1}{n+1},\frac{1}{n}))$ for $n\in\{N,N+1,N+2,\cdots\}$. Left-multiplying \eqref{syequ} by $\mathcal{P}_{n}$ and right-multiplying it by $\mathcal{P}_{m}$ gives $\lambda^2\mathcal{P}_{n}|T|^\frac{1}{2}C\mathcal{P}_{m}+\bigl(1-(1-\lambda)^2\bigr)\mathcal{P}_{n}C|T|^\frac{1}{2}\mathcal{P}_{m}=0$; then, by the idempotency of $\mathcal{P}_{n}$ and its commutativity with $|T|$, we arrive at the following:
		\begin{equation}\label{sylvester}
		(\lambda^2\mathcal{P}_{n}|T|^\frac{1}{2}\mathcal{P}_{n})(\mathcal{P}_{n}C\mathcal{P}_{m})-(\mathcal{P}_{n}C\mathcal{P}_{m})\bigl(\bigl((1-\lambda)^2-1\bigr)\mathcal{P}_{m}|T|^\frac{1}{2}\mathcal{P}_{m}\bigr)=0
		\end{equation}
		for $n,m\in\{0,N,N+1,N+2,\cdots\}\triangleq I$. Now, we regard $\lambda^2\mathcal{P}_{n}|T|^\frac{1}{2}\mathcal{P}_{n}$, $((1-\lambda)^2-1)\mathcal{P}_{m}|T|^\frac{1}{2}\mathcal{P}_{m}$ and $\mathcal{P}_{n}C\mathcal{P}_{m}$ as operators on the Hilbert spaces $\mathcal{P}_{n}\mathcal{H}$, $\mathcal{P}_{m}\mathcal{H}$, and from $\mathcal{P}_{m}\mathcal{H}$ to $\mathcal{P}_{n}\mathcal{H}$, respectively. Note that
		\begin{align*}
			\sigma(\lambda^2\mathcal{P}_{n}|T|^\frac{1}{2}\mathcal{P}_{n})&=\{\lambda^2\cdot\gamma:\gamma\in\sigma(\mathcal{P}_{n}|T|^\frac{1}{2}\mathcal{P}_{n})\}\\&\subseteq\begin{cases}
				\{0\}&n=0\\
				[\frac{\lambda^2}{n},+\infty)&n\in I\backslash\{0\}
			\end{cases}
		\end{align*}
		and
		\begin{align*}
			\sigma\bigl(\bigl((1-\lambda)^2-1\bigr)\mathcal{P}_{m}|T|^\frac{1}{2}\mathcal{P}_{m}\bigr)&=\{((1-\lambda)^2-1)\cdot\beta:\beta\in\sigma(\mathcal{P}_{m}|T|^\frac{1}{2}\mathcal{P}_{m})\}\\&\subseteq\begin{cases}
				\{0\}&m=0\\
				(-\infty,\frac{(1-\lambda)^2-1}{m}]&m\in I\backslash\{0\}
			\end{cases},
		\end{align*}
		which implies $\sigma(\lambda^2\mathcal{P}_{n}|T|^\frac{1}{2}\mathcal{P}_{n})\bigcap\sigma\bigl(\bigl((1-\lambda)^2-1\bigr)\mathcal{P}_{m}|T|^\frac{1}{2}\mathcal{P}_{m}\bigr)=\emptyset$ for every $(n,m)\in (I\times I)\backslash\{(0,0)\}$. Applying Lemma \ref{dukeequ} with $A=\lambda^2\mathcal{P}_{n}|T|^\frac{1}{2}\mathcal{P}_{n}$, $B=((1-\lambda)^2-1)\mathcal{P}_{m}|T|^\frac{1}{2}\mathcal{P}_{m}$ and $S=0$, from \eqref{sylvester} we obtain $\mathcal{P}_{n}C\mathcal{P}_{m}=0$ for $(n,m)\in (I\times I)\backslash\{(0,0)\}$. Combining the fact that $C\mathcal{P}_{0}=CP_{\mathcal{N}(|T|)}=|T|^\frac{1}{2}UP_{\mathcal{N}(|T|)}-U|T|^\frac{1}{2}P_{\mathcal{N}(|T|)}=|T|^\frac{1}{2}UP_{\mathcal{N}(U)}-0=0$, we conclude that $$\mathcal{P}_{n}C\mathcal{P}_{m}=0$$ for $n,m\in I$. Using the above equality together with the strong operator convergence of $\mathcal{P}_{0}+\lim\limits_{l\to\infty}\sum\limits_{k=N}^l\mathcal{P}_{k}$ to the identity operator, we can easily verify that $C=0$. The proof is complete.
	\end{proof}

In order to characterize the binormality via the power mean transform, the following lemmas are needed.
	\begin{lemma}\label{commu}
		Let $T\in\mathbb{B}(\mathcal{H})$. Suppose $S\in \mathbb{B}(\mathcal{H})$ is positive. Then $[S,T]=0$ if and only if $[S^\alpha,T]=0$ for all $\alpha>0$.
	\end{lemma}
	\begin{proof}
	This is a direct consequence of the continuous functional calculus.
	\end{proof}
	\begin{lemma}\cite{fur-asm-1983}\label{commu8}
Suppose $S,T\in\mathbb{B}(\mathcal{H})$ are positive and $[S,T]=0$. Then $[P_{\mathcal{N}(S)^\perp},P_{\mathcal{N}(T)^\perp}]=[P_{\mathcal{N}(S)^\perp},T]=[S,P_{\mathcal{N}(T)^\perp}]=0$.
	\end{lemma}
	Our main result is now presented.
	
	\begin{theorem}\label{commu3}
		Let $T=U|T|$ be the polar decomposition of $T\in\mathbb{B}(\mathcal{H})$. Then the following conditions are equivalent:
		\begin{itemize}
			\item[(i)] $T$ is binormal.
			\item[(ii)] $UU^\ast|T|=|T|UU^\ast$ and $U^\ast\Delta(T)$ is self-adjoint.
			\item[(iii)] for all $\lambda\in(0,1)$, $P_{\lambda}(T)=U|P_{\lambda}(T)|$ is the polar decomposition of $P_{\lambda}(T)$, where $|P_{\lambda}(T)|=((1-\lambda)|T|^\frac{1}{2}+\lambda U^\ast|T|^{\frac{1}{2}}U)^2$.
			\item[(iv)] for some $\lambda\in(0,1)$, $P_{\lambda}(T)=U|P_{\lambda}(T)|$ is the polar decomposition of $P_{\lambda}(T)$, where $|P_{\lambda}(T)|=((1-\lambda)|T|^\frac{1}{2}+\lambda U^\ast|T|^{\frac{1}{2}}U)^2$. 
		\end{itemize}
	\end{theorem}
	\begin{proof}
		First we prove (i) $\Leftrightarrow$ (ii); then we prove (ii) $\Rightarrow$ (iii) $\Rightarrow$ (iv) $\Rightarrow$ (ii).	
		
		(i) $\Leftrightarrow$ (ii): Suppose $T$ is binormal, then $[T^\ast T,TT^\ast]=0$. By Lemma \ref{commu} we have $[|T|,|T^\ast|]=0$, which, by $|T^\ast|=U|T|U^\ast$, implies $U^\ast[|T|,U|T|U^\ast]U=0$. Combining
		\begin{equation}\label{alucom1}
	U^\ast U|T|=P_{\overline{\mathcal{R}(|T|)}}|T|=|T|=|T|U^\ast U
		\end{equation}
	 and Lemma \ref{commu}, we obtain
		\begin{equation}\label{commu1}
		[U^\ast|T|U,|T|^\frac{1}{2}]=0.
		\end{equation} 
		On the other hand, since $UU^\ast=P_{\overline{\mathcal{R}(|T^\ast|)}}=P_{\mathcal{N}(|T^\ast|)^\perp}$, \begin{equation}\label{pro}
				|T|UU^\ast=UU^\ast|T|
			\end{equation} follows from Lemma \ref{commu8}. Next, if we can prove $(U^\ast|T|^\frac{1}{2}U)^2=U^\ast|T|U$, then by the uniqueness of the positive square root together with Lemma \ref{commu} (applied to Equation \eqref{commu1}), we have
			\begin{equation}\label{alucom}
			[U^\ast|T|^\frac{1}{2}U,|T|^\frac{1}{2}]=0,
			\end{equation}
			 i.e., $U^\ast\Delta(T)$ is self-adjoint. Indeed, combining \eqref{pro} and Lemma \ref{commu}, we get
		\begin{equation}\label{squareroot}
			|T|^\frac{1}{2}UU^\ast=UU^\ast|T|^\frac{1}{2}.
		\end{equation}
		Then we deduce\begin{align*}
			(U^\ast|T|^\frac{1}{2}U)^2=U^\ast|T|^\frac{1}{2}UU^\ast|T|^\frac{1}{2}U=U^\ast UU^\ast|T|^\frac{1}{2}|T|^\frac{1}{2}U=P_{\overline{\mathcal{R}(|T|)}}U^\ast|T|U=U^\ast|T|U.
		\end{align*}
		The last equality is valid since $\mathcal{R}(U^\ast)=\overline{\mathcal{R}(|T|)}$. Hence the desired result follows.
		
		Conversely, if $U^\ast\Delta(T)$ is self-adjoint, then by Lemma \ref{commu} we have
	\begin{equation}\label{squareroot1}
	[U^\ast|T|^\frac{1}{2}U,|T|]=0
		\end{equation}
On the other hand, using Lemma \ref{commu} and the condition $UU^\ast|T|=|T|UU^\ast$ yields $UU^\ast|T|^\frac{1}{2}=|T|^\frac{1}{2}UU^\ast$. Then the same proof as above gives
		\begin{equation}\label{squareroot2}
			(U^\ast|T|^\frac{1}{2}U)^2=U^\ast|T|U.
		\end{equation}
		By \eqref{squareroot2} and Lemma \ref{commu} applied to \eqref{squareroot1}, we obtain
		$U[U^\ast|T|U,|T|]U^\ast=0$. Further, together with $|T^\ast|=U|T|U^\ast$ and the condition $UU^\ast|T|=|T|UU^\ast$, we get $$|T|(UU^\ast|T^\ast|)=(|T^\ast|UU^\ast)|T|,$$ which, by $UU^\ast=P_{\overline{\mathcal{R}(|T^\ast|)}}$, implies $$|T|\Big((P_{\overline{\mathcal{R}(|T^\ast|)}}+P_{\mathcal{N}(|T^\ast|)})|T^\ast|\Big)=\Big(|T^\ast|\Big(P_{\overline{\mathcal{R}(|T^\ast|)}}+P_{\mathcal{N}(|T^\ast|)}\Big)\Big)|T|,$$
		i.e., $|T||T^\ast|=|T^\ast||T|$. Thus, (i) is equivalent to (ii).
		
		(ii) $\Rightarrow$ (iii): Let $\lambda\in(0,1)$. Write the positive operator $\widehat{P}=((1-\lambda)|T|^\frac{1}{2}+\lambda U^\ast|T|^\frac{1}{2}U)^2$. We need only prove $P_{\lambda}(T)=U\widehat{P}$ and $\mathcal{N}(U)=\mathcal{N}(\widehat{P})$; then the desired result follows from the Polar Decomposition Theorem in \cite[Chapter VIII, \S 3]{Conway-book-1990}. Indeed, a direct computation shows that \begin{align*}
			U\widehat{P}&=U((1-\lambda)|T|^\frac{1}{2}+\lambda U^\ast|T|^\frac{1}{2}U)^2\\&=U((1-\lambda)^2|T|+\lambda(1-\lambda)|T|^\frac{1}{2}U^\ast|T|^\frac{1}{2}U+\lambda(1-\lambda)U^\ast|T|^\frac{1}{2}U|T|^\frac{1}{2}+\lambda^2U^\ast|T|^\frac{1}{2}UU^\ast|T|^\frac{1}{2}U)\\&\xlongequal{U^\ast\Delta(T)\text{ is self-adjoint}}U((1-\lambda)^2|T|+2\lambda(1-\lambda)U^\ast|T|^\frac{1}{2}U|T|^\frac{1}{2}+\lambda^2U^\ast|T|^\frac{1}{2}UU^\ast|T|^\frac{1}{2}U)\\&=(1-\lambda)^2U|T|+2\lambda(1-\lambda)UU^\ast|T|^\frac{1}{2}U|T|^\frac{1}{2}+\lambda^2UU^\ast|T|^\frac{1}{2}UU^\ast|T|^\frac{1}{2}U\\&\xlongequal{\eqref{squareroot}}(1-\lambda)^2U|T|+2\lambda(1-\lambda)|T|^\frac{1}{2}UU^\ast U|T|^\frac{1}{2}+\lambda^2|T|^\frac{1}{2}UU^\ast UU^\ast|T|^\frac{1}{2}U\\&\xlongequal{UU^\ast U=U}(1-\lambda)^2T+2\lambda(1-\lambda)|T|^\frac{1}{2}U|T|^\frac{1}{2}+\lambda^2|T|^\frac{1}{2}UU^\ast|T|^\frac{1}{2}U\\&\xlongequal{\eqref{squareroot}}(1-\lambda)^2T+2\lambda(1-\lambda)|T|^\frac{1}{2}U|T|^\frac{1}{2}+\lambda^2|T|^\frac{1}{2}|T|^\frac{1}{2}UU^\ast U\\&\xlongequal{UU^\ast U=U}(1-\lambda)^2T+2\lambda(1-\lambda)|T|^\frac{1}{2}U|T|^\frac{1}{2}+\lambda^2|T|U\\&=(1-\lambda)^2T+2\lambda(1-\lambda)\Delta(T)+\lambda^2T^D=P_{\lambda}(T).
		\end{align*}
		On the other hand, if $x\in\mathcal{N}((1-\lambda)|T|^\frac{1}{2}+\lambda U^\ast|T|^\frac{1}{2}U)=\mathcal{N}((\widehat{P})^\frac{1}{2})=\mathcal{N}(\widehat{P})$, then it follows from the positivity of $(1-\lambda)|T|^\frac{1}{2}$ and $\lambda U^\ast|T|^\frac{1}{2}U$ that $(1-\lambda)|T|^\frac{1}{2}x=0\Rightarrow x\in\mathcal{N}(|T|^\frac{1}{2})=\mathcal{N}(|T|)=\mathcal{N}(U)$. Hence, $\mathcal{N}(\widehat{P})\subseteq\mathcal{N}(U)$ follows. Finally, from $\mathcal{N}(U)=\mathcal{N}(|T|^\frac{1}{2})$ we have $\mathcal{N}(U)\subseteq\mathcal{N}((\widehat{P})^\frac{1}{2})=\mathcal{N}(\widehat{P})$. Hence, (iii) holds.
		
		(iii) $\Rightarrow$ (iv) is clear.
		
		(iv) $\Rightarrow$ (ii): Suppose $P_{\lambda}(T)=U|P_{\lambda}(T)|$ is the polar decomposition of $P_{\lambda}(T)$ for some $\lambda\in(0,1)$. It follows that
		\begin{align}\label{self3}
			\nonumber
			|P_{\lambda}(T)|=U^\ast U|P_{\lambda}(T)|&=U^\ast P_{\lambda}(T)=U^\ast((1-\lambda)^2T+\lambda^2T^D+2\lambda(1-\lambda)\Delta(T))\\&=(1-\lambda)^2|T|+\lambda^2U^\ast|T|U+2\lambda(1-\lambda)U^\ast|T|^\frac{1}{2}U|T|^\frac{1}{2}.
		\end{align}
		By taking the adjoint of the above equation, we easily get $U^\ast|T|^\frac{1}{2}U|T|^\frac{1}{2}=|T|^\frac{1}{2}U^\ast|T|^\frac{1}{2}U$, i.e., $U^\ast\Delta(T)$ is self-adjoint. 
		
		On the other hand, expanding the condition $|P_{\lambda}(T)|=((1-\lambda)|T|^\frac{1}{2}+\lambda U^\ast|T|^{\frac{1}{2}}U)^2$ yields \begin{align}\label{self4}
			\nonumber
			|P_{\lambda}(T)|&=(1-\lambda)^2|T|+\lambda(1-\lambda)|T|^\frac{1}{2}U^\ast|T|^\frac{1}{2}U+\lambda(1-\lambda)U^\ast|T|^\frac{1}{2}U|T|^\frac{1}{2}+\lambda^2U^\ast|T|^\frac{1}{2}UU^\ast|T^\frac{1}{2}|U\\&=(1-\lambda)^2|T|+2\lambda(1-\lambda)U^\ast|T|^\frac{1}{2}U|T|^\frac{1}{2}+\lambda^2U^\ast|T|^\frac{1}{2}UU^\ast|T^\frac{1}{2}|U.
		\end{align}
		From \eqref{self3} and \eqref{self4} we obtain $$U^\ast|T|^\frac{1}{2}UU^\ast|T|^\frac{1}{2}U=U^\ast|T|U.$$
		Left-multiplying by $U$ and right-multiplying by $U^\ast$ the above equation gives
		$$UU^\ast|T|^\frac{1}{2}UU^\ast|T|^\frac{1}{2}UU^\ast=UU^\ast|T|UU^\ast.$$
		Then for any $x\in\mathcal{R}(UU^\ast)$ we have $UU^\ast|T|^\frac{1}{2}(1-UU^\ast)|T|^\frac{1}{2}x=0$, which implies, by $(1-UU^\ast)$ being a projection, that 
		\begin{align*}
		&\langle|T|^\frac{1}{2}(1-UU^\ast)|T|^\frac{1}{2}x,UU^\ast x\rangle=0\\&\Leftrightarrow\langle|T|^\frac{1}{2}(1-UU^\ast)|T|^\frac{1}{2}x,x\rangle=0\\&\Leftrightarrow\langle(1-UU^\ast)|T|^\frac{1}{2}x,(1-UU^\ast)|T|^\frac{1}{2}x\rangle=0\\&\Leftrightarrow (1-UU^\ast)|T|^\frac{1}{2}x=0.
		\end{align*}
		Thus, $|T|^\frac{1}{2}x\in\mathcal{R}(UU^\ast)$, hence $UU^\ast|T|^\frac{1}{2}UU^\ast=|T|^\frac{1}{2}UU^\ast$. Then, using the self-adjointness of $UU^\ast|T|^\frac{1}{2}UU^\ast$ and Lemma \ref{commu}, we conclude $$|T|UU^\ast=UU^\ast|T|.$$
		The proof is complete. 
	\end{proof}
	\begin{remark}
		(i) The equality $UU^\ast|T|=|T|UU^\ast$ cannot be omitted in condition (ii) of the above theorem. Consider the matrix $T= \begin{bmatrix}
			\frac{\sqrt{2}}{2}&0\\
			\frac{\sqrt{2}}{2}&0
		\end{bmatrix}$. It has the polar decomposition $T=U|T|$, where $U= \begin{bmatrix}
			\frac{\sqrt{2}}{2}&0\\
			\frac{\sqrt{2}}{2}&0
		\end{bmatrix}$ and $|T|= \begin{bmatrix}
			1&0\\
			0&0
		\end{bmatrix}$. Then we have $$U^\ast\Delta(T)=U^\ast|T|^\frac{1}{2}U|T|^\frac{1}{2}=\begin{bmatrix}
			\frac{\sqrt{2}}{2}&\frac{\sqrt{2}}{2}\\
			0&0
		\end{bmatrix}\begin{bmatrix}
			1&0\\
			0&0
		\end{bmatrix}\begin{bmatrix}
			\frac{\sqrt{2}}{2}&0\\
			\frac{\sqrt{2}}{2}&0
		\end{bmatrix}=\begin{bmatrix}
			\frac{1}{2}&0\\
			0&0
		\end{bmatrix},$$
		$$TT^\ast T^\ast T=\begin{bmatrix}
			\frac{1}{2}&\frac{1}{2}\\
			\frac{1}{2}&\frac{1}{2}
		\end{bmatrix}\begin{bmatrix}
			1&0\\
			0&0
		\end{bmatrix}=\begin{bmatrix}
			\frac{1}{2}&0\\
			\frac{1}{2}&0
		\end{bmatrix}$$
		and
		$$T^\ast T T T^\ast=\begin{bmatrix}
			1&0\\
			0&0
		\end{bmatrix}\begin{bmatrix}
			\frac{1}{2}&\frac{1}{2}\\
			\frac{1}{2}&\frac{1}{2}
		\end{bmatrix}=\begin{bmatrix}
			\frac{1}{2}&\frac{1}{2}\\
			0&0
		\end{bmatrix}.$$
		Hence, $U^\ast\Delta(T)$ is self-adjoint but not binormal for $T$. 
		
		(ii) If $\lambda=1$, then conditions (i) and (iv) in the above theorem are not equivalent in general. Indeed, consider the matrix $T$ defined in Remark \ref{duggal}. A simple calculation shows that$$TT^\ast T^\ast T=\begin{bmatrix}
			1&0\\
			0&0
		\end{bmatrix}\begin{bmatrix}
			0&0\\
			0&1
		\end{bmatrix}=\begin{bmatrix}
			0&0\\
			0&0
		\end{bmatrix}=\begin{bmatrix}
			0&0\\
			0&1
		\end{bmatrix}\begin{bmatrix}
			1&0\\
			0&0
		\end{bmatrix}=T^\ast TTT^\ast,$$
		that is, $T$ is binormal. However, it follows easily from Remark \ref{duggal} that the partial isometry in the polar decomposition of $T$ is $\begin{bmatrix}
			0&1\\
			0&0
		\end{bmatrix}$, while that of $P_{\lambda}(T)=T^D$ is $0$. Therefore, the condition (i) does not imply (iv). On the other hand, consider the polar decomposition $S=U|S|$ of $S=\frac{1}{\sqrt{2}}\begin{bmatrix}
			1&-2&0\\
			1&2&0\\
			0&0&0
		\end{bmatrix}$, where $U=\frac{1}{\sqrt{2}}\begin{bmatrix}
			1&-1&0\\
			1&1&0\\
			0&0&0
		\end{bmatrix}$ and $|S|=\begin{bmatrix}
			1&0&0\\
			0&2&0\\
			0&0&0
		\end{bmatrix}$. Direct verification gives that $P_{\lambda}(S)=S^D=|S|U=\begin{bmatrix}
			1&0&0\\
			0&2&0\\
			0&0&0
		\end{bmatrix}\begin{bmatrix}
			\frac{1}{\sqrt{2}}&\frac{-1}{\sqrt{2}}&0\\
			\frac{1}{\sqrt{2}}&\frac{1}{\sqrt{2}}&0\\
			0&0&0
		\end{bmatrix}=\begin{bmatrix}
			\frac{1}{\sqrt{2}}&\frac{-1}{\sqrt{2}}&0\\
			\sqrt{2}&\sqrt{2}&0\\
			0&0&0
		\end{bmatrix}$ has the polar decomposition $S^D=\begin{bmatrix}
			\frac{1}{\sqrt{2}}&\frac{-1}{\sqrt{2}}&0\\
			\frac{1}{\sqrt{2}}&\frac{1}{\sqrt{2}}&0\\
			0&0&0
		\end{bmatrix}\begin{bmatrix}
			\frac{3}{2}&\frac{1}{2}&0\\
			\frac{1}{2}&\frac{3}{2}&0\\
			0&0&0
		\end{bmatrix}=U|S^D|=U(U^\ast|S|^\frac{1}{2}U)^2$. Further, we have $|S^\ast|=U|S|U^\ast=\frac{1}{2}\begin{bmatrix}
			3&-1&0\\
			-1&3&0\\
			0&0&0
		\end{bmatrix}$, which yields that $$|S||S^\ast|=\frac{1}{2}\begin{bmatrix}
			3&-1&0\\
			-2&6&0\\
			0&0&0
		\end{bmatrix}\ne\frac{1}{2}\begin{bmatrix}
			3&-2&0\\
			-1&6&0\\
			0&0&0
		\end{bmatrix}=|S^\ast||S|,$$ that is, $S$ is not binormal, and so (iv) does not imply (i). Finally, it should be noted that, in view of Lemma 3.4 and Theorem 3.7 in \cite{Mat-Ope-2009}, the implication (i) $\Rightarrow$ (iv) holds when $T$ is invertible and binormal with $\lambda=1$.
	\end{remark}
	
	As a consequence of the above theorem, we characterize when a binormal operator is normal via the power mean transform.
	
	\begin{corollary}\label{normalpower}
	Let $T\in\mathbb{B}(\mathcal{H})$ be binormal. Then $T$ is normal if and only if $\mathcal{N}(T^\ast)\subseteq\mathcal{N}(T)$ and $P_{\frac{1}{2}}(T)=\Delta(T)$.
	\end{corollary}
	
	Before proving this corollary, the following example shows that when $\lambda\ne\frac{1}{2}$, even if $T$ is a binormal operator satisfying $\mathcal{N}(T^\ast)\subseteq\mathcal{N}(T)$ and $P_{\lambda}(T)=\Delta(T)$, $T$ need not be normal; hence the condition $\lambda=\frac{1}{2}$ is necessary. Moreover, the equation $P_{\frac{1}{2}}(T)=\Delta(T)$ does not necessarily imply $\mathcal{N}(T^\ast)\subseteq\mathcal{N}(T)$.
	\begin{example}
	Let $\lambda\in(0,1)$ and $W_{\alpha}$ be the bilateral weighted shift on $l^{2}(\mathbb{Z})$ defined by $W_{\alpha}e_{n}=\alpha_{n}e_{n+1}$ for $n\in\mathbb{Z}$, where $\alpha_{n}=\begin{cases}
		1&n\le0\\
		(\frac{1-\lambda}{\lambda})^4&n\ge1
	\end{cases}.$ Then $W_{\alpha}$ has the polar decomposition $W_{\alpha}=U|W_{\alpha}|$, where $Ue_{n}=e_{n+1}$ and $|W_{\alpha}|e_{n}=\alpha_{n}e_{n}$ for $n\in\mathbb{Z}$. We can easily obtain that $\mathcal{N}(W_{\alpha}^\ast)=\mathcal{N}(U^\ast)=\mathcal{N}(U)=\mathcal{N}(W_{\alpha})=\{0\}$. On the other hand, from \eqref{meanuni} we have 
	\begin{align*}
		P_{\lambda}(W_{\alpha})e_{n}&=((1-\lambda)\sqrt{\alpha_{n}}+\lambda\sqrt{\alpha_{n+1}})^2e_{n+1}\\&=\widetilde{\alpha_{n}}e_{n+1}\\&=\sqrt{\alpha_{n}\alpha_{n+1}}e_{n+1}=|W_{\alpha}|^\frac{1}{2}U|W_{\alpha}|^\frac{1}{2}e_{n}=\Delta(W_{\alpha})e_{n}
	\end{align*}
	where $\widetilde{\alpha_{n}}=\begin{cases}
		1&n\le-1\\(\frac{1-\lambda}{\lambda})^2&n=0\\(\frac{1-\lambda}{\lambda})^4&n\ge1
	\end{cases}$. However, one readily finds that $W_{\alpha}$ is normal if and only if $\lambda=\frac{1}{2}$. Thus, we know that for $\lambda\ne\frac{1}{2}$, $W_{\alpha}$ satisfies $\mathcal{N}(W_{\alpha}^\ast)\subseteq\mathcal{N}(W_{\alpha})$ and $P_{\lambda}(T)=\Delta(T)$ but $W_{\alpha}$ is not normal. Finally, we may consider the unilateral shift operator $T=\operatorname{shift}(1,1,1,\cdots)$; it is easy to verify that $P_{\frac{1}{2}}(T)=\Delta(T)=T$ but $\mathcal{N}(T)=\{0\}\ne\mathcal{N}(T^\ast)$, i.e., $\mathcal{N}(T^\ast)\nsubseteq\mathcal{N}(T)$. We conclude that $P_{\frac{1}{2}}(T)=\Delta(T)$ does not imply $\mathcal{N}(T^\ast)\subseteq\mathcal{N}(T)$.
	\end{example}
	Now, we prove Corollary \ref{normalpower}
	\begin{proof}[Proof of Corollary \ref{normalpower}]
	\textit{Sufficiency}. Since $T$ is normal, it is quasinormal, so $P_{\frac{1}{2}}(T)=\Delta(T)=T$. The inclusion $\mathcal{N}(T^\ast)\subseteq\mathcal{N}(T)$ is obvious, as $\|Tx\|=\|T^\ast x\|$ for all $x$.

	\textit{Necessity}. Since $T$ is binormal, Theorem \ref{commu3} (ii) gives that
	$U^\ast\Delta(T)$ is self-adjoint, i.e., $[|T|^\frac{1}{2},U^\ast|T|^\frac{1}{2}U]=0$. Using $UU^\ast|T|=|T|UU^\ast$ and $UU^\ast U=U$, a direct computation yields 
	\begin{align*}
	|\Delta(T)|=|T|^\frac{1}{2}U^\ast|T|^\frac{1}{2}U.
	\end{align*}
The hypothesis $P_{\frac{1}{2}}(T)=\Delta(T)$ implies $|P_{\frac{1}{2}}(T)|=|\Delta(T)|$. By Theorem \ref{commu3} (iii), $$(\frac{1}{2}|T|^\frac{1}{2}+\frac{1}{2}U^\ast|T|^\frac{1}{2}U)^2=|T|^\frac{1}{2}U^\ast|T|^\frac{1}{2}U,$$ which is equivalent to $(|T|^\frac{1}{2}-U^\ast|T|^\frac{1}{2}U)^2=0$. Hence, $$|T|^\frac{1}{2}=U^\ast|T|^\frac{1}{2}U.$$
Multiplying the above equality on the right side by $U^\ast$ and using $U^\ast U U^\ast=U^\ast$ together with $UU^\ast|T|=|T|UU^\ast$ we obtain $|T|U^\ast=U^\ast|T|$. Then by Lemma \ref{commu}, $|T|^2U^\ast=U^\ast|T|^2$. Consequently, $$TT^\ast=U|T|^2U^\ast=UU^\ast|T|^2=|T|^2=T^\ast T,$$
where the third equality follows from $\mathcal{N}(T^\ast)\subseteq\mathcal{N}(T)$ (see \cite[Lemma 2.14]{C.C.M}). The proof is complete.
	\end{proof}
	
An operator $T\in\mathbb{B}(\mathcal{H})$ is said to be {\em $p$-hyponormal} (resp. {\em co-$p$-hyponormal}) if $(T^\ast T)^p\ge(TT^\ast)^p$ (resp. $(T^\ast T)^p\le(TT^\ast)^p$). In particular, when $p=1$, $T$ is called {\em hyponormal} (resp. {\em co-hyponormal}; when $p=\frac{1}{2}$, $T$ is called {\em semi-hyponormal} (resp. {\em co-semi-hyponormal}). Chabbabi proved in \cite[Theorem 2.9]{chabbabi-pams-2023} that if $T$ is co-hyponormal and binormal, then $M(T)$ is co-hyponormal. Below we present some results close to those of Chabbabi and give a negative answer to Question 10 in \cite{yamazaki-laa-2023}.	
	\begin{corollary}\label{hy10}
	Let $\lambda\in(0,1)$. Let $T$ be a binormal operator with the polar decomposition $T=U|T|$. Then 
	\begin{itemize}
	\item[(i)] $(P_{\lambda}(T))^\ast P_{\lambda}(T)\le T^\ast T$ if and only if $U^\ast|T|U\le|T|$.
	\item[(ii)] If $|T|\ge|T^\ast|$, then $|(P_{\lambda}(T))^\ast|\le|T|\le|P_{\lambda}(T)|$.
	\end{itemize}
	\end{corollary}
	\begin{proof}
	(i) By the Loewner-Heinz inequality and \eqref{alucom}, we have $$(P_{\lambda}(T))^\ast P_{\lambda}(T)\le T^\ast T\Leftrightarrow|P_{\lambda}(T)|^\frac{1}{2}\le|T|^\frac{1}{2}.$$
	Combining Theorem \ref{commu3} with \eqref{alucom} and \eqref{squareroot2}, the latter inequality is clearly equivalent to $U^\ast|T|U\le|T|$.

(ii) By \eqref{alucom} and \eqref{squareroot}, we have
\begin{equation}\label{hy1}
[U|T|^\frac{1}{2}U^\ast,|T|^\frac{1}{2}]=0.
\end{equation}
  From $|T^\ast|=U|T|U^\ast$ and $|T|\ge|T^
\ast|$, the Loewner-Heinz inequality yields
\begin{equation}\label{hy2}
|T|^\frac{1}{2}\ge U|T|^\frac{1}{2}U^\ast.
\end{equation}
Combining \eqref{hy1}, \eqref{hy2}, \eqref{squareroot} and \eqref{alucom} gives
\begin{equation}\label{hy5}
	|T|=|T|^\frac{1}{2}|T|^\frac{1}{2}\ge U|T|^\frac{1}{2}U^\ast|T|^\frac{1}{2}=U|T|^\frac{1}{2}U^\ast|T|^\frac{1}{2}U U^\ast=UU^\ast|T|^\frac{1}{2}U|T|^\frac{1}{2}U^\ast.
\end{equation}
Theorem \ref{commu3} implies
\begin{equation}\label{hy6}
\begin{aligned}
	|T|-UU^\ast|T|UU^\ast&=|T|-|T|UU^\ast-UU^\ast|T|+UU^\ast|T|UU^\ast\\&=(|T|-UU^\ast|T|)(1-UU^\ast)=(1-UU^\ast)|T|(1-UU^\ast)\ge0.
\end{aligned}	
\end{equation}
From $|T|\ge U|T|U^\ast$ and \eqref{alucom1} we get
\begin{equation}\label{hy3}
U^\ast|T|U\ge|T|
\end{equation}
and hence, by the Loewner-Heinz inequality and \eqref{squareroot2}, \begin{equation}\label{hy4}
U^\ast|T|^\frac{1}{2}U\ge|T|^\frac{1}{2}.
\end{equation}
Therefore, using \eqref{alucom}, \eqref{hy3}, \eqref{hy4}, \eqref{hy5} and \eqref{hy6} we obtain
\begin{align*}
	|P_{\lambda}(T)|&=(1-\lambda)^2|T|+\lambda^2U^\ast|T|U+2\lambda(1-\lambda)U^\ast|T|^\frac{1}{2}U|T|^\frac{1}{2}\\&\ge(1-\lambda)^2|T|+\lambda^2|T|+2\lambda(1-\lambda)|T|\\&\ge|T|\\&\ge(1-\lambda)^2U|T|U^\ast+\lambda^2UU^\ast|T|UU^\ast+2\lambda(1-\lambda)UU^\ast|T|^\frac{1}{2}U|T|^\frac{1}{2}U^\ast\\&=U|P_{\lambda}(T)|U^\ast=|(P_{\lambda}(T))^\ast|,
\end{align*}as desired.	
	\end{proof}
\begin{remark}\label{rkbin}
(i) Let $T=U|T|$ be the polar decomposition. Golla et al. asked in \cite[Question 10]{yamazaki-laa-2023} whether $(\Delta_{\mathsf{m}_{f}}(T))^\ast\Delta_{\mathsf{m}_{f}}(T)\le T^\ast T$ holds whenever $UU^\ast|T|=|T|$. For invertible $T$, we have $P_{\lambda}(T)=\Delta_{\mathsf{m}_{f}}(T)$ with $f(x)=(\lambda+(1-\lambda)\sqrt{x})^2$ on $(0,\infty)$. Hence Corollary \ref{hy10} (i) yields that for an invertible binormal operator $T$, $(\Delta_{\mathsf{m}_{f}}(T))^\ast\Delta_{\mathsf{m}_{f}}(T)\le T^\ast T\Leftrightarrow U^\ast|T|U\le|T|\Leftrightarrow|T|\le|T^\ast|\Leftrightarrow T$ is co-semi-hyponormal. Thus, if $T$ is invertible, binormal, but not co-semi-hyponormal, the inequality fails. Such an operator exists, for example the bilateral weighted shift $W_{\alpha}$ on $l^2(\mathbb{Z})$ with $\alpha_{n}=\begin{cases}
2&n\le0\\
1&n\ge1
\end{cases}$. Therefore, the answer to the question of Golla et al. is negative.

(ii) Corollary \ref{hy10} (ii) shows that for a binormal semi-hyponormal operator $T$, its power mean transform $P_{\lambda}(T)$ remains semi-hyponormal for $\lambda\in(0,1)$.
\end{remark}	
	
	The following example shows that if $T$ is binormal, then for $\lambda\in(0,1)$ the power mean transform $P_{\lambda}(T)$ may fail to be binormal.
	\begin{example}\label{exmaple1}
		Let $T$ be the matrix $\begin{bmatrix}
			0&0&5\\
			\frac{1}{2}&\frac{\sqrt{3}}{2}&0\\
			\frac{\sqrt{3}}{2}&\frac{-1}{2}&0
		\end{bmatrix}$ with the polar decomposition $T=U|T|$, where 
		$|T|=\begin{bmatrix}
			1&0&0\\
			0&1&0\\
			0&0&5
		\end{bmatrix}$ and $U=T|T|^{-1}=\begin{bmatrix}
			0&0&1\\
			\frac{1}{2}&\frac{\sqrt{3}}{2}&0\\
			\frac{\sqrt{3}}{2}&\frac{-1}{2}&0
		\end{bmatrix}$. Then we have $$|T||T^\ast|=|T|U|T|U^\ast=\begin{bmatrix}
			5&0&0\\
			0&1&0\\
			0&0&5
		\end{bmatrix}=U|T|U^\ast|T|=|T^\ast||T|,$$i.e., $T$ is binormal. Let $\lambda\in(0,1)$. It follows from Theorem \ref{commu3} (iii) that $P_{\lambda}(T)=U|P_{\lambda}(T)|$ is the polar decomposition of $P_{\lambda}(T)$ and
		\begin{align*}
			|P_{\lambda}(T)|^\frac{1}{2}=(1-\lambda)|T|^\frac{1}{2}+\lambda U^\ast|T|^\frac{1}{2}U&=(1-\lambda)\begin{bmatrix}
				1&0&0\\
				0&1&0\\
				0&0&\sqrt{5}
			\end{bmatrix}+\lambda\begin{bmatrix}
				\frac{1+3\sqrt{5}}{4}&\frac{\sqrt{3}-\sqrt{15}}{4}&0\\
				\frac{\sqrt{3}-\sqrt{15}}{4}&\frac{3+\sqrt{5}}{4}&0\\
				0&0&1
			\end{bmatrix}\\&=\begin{bmatrix}
				(1-\lambda)+\lambda\frac{1+3\sqrt{5}}{4}&\lambda\frac{\sqrt{3}-\sqrt{15}}{4}&0\\
				\lambda\frac{\sqrt{3}-\sqrt{15}}{4}&(1-\lambda)+\lambda\frac{3+\sqrt{5}}{4}&0\\
				0&0&(1-\lambda)\sqrt{5}+\lambda
			\end{bmatrix}.
		\end{align*}
		Thus, a direct computation gives
		\begin{align*}
			|P_{\lambda}(T)|^\frac{1}{2}|(P_{\lambda}(T))^\ast|^\frac{1}{2}&=|P_{\lambda}(T)|^\frac{1}{2}U|P_{\lambda}(T)|^\frac{1}{2}U^\ast=\begin{bmatrix}
				*&*&*\\
				\frac{\sqrt{3}-\sqrt{15}}{4}\lambda(\sqrt{5}(1-\lambda)+\lambda)&*&*\\
				*&*&*
			\end{bmatrix}\\&\ne\begin{bmatrix}
				*&*&*\\
				\frac{\sqrt{3}-\sqrt{15}}{4}\lambda&*&*\\
				*&*&*
			\end{bmatrix}=U|P_{\lambda}(T)|^\frac{1}{2}U^\ast|P_{\lambda}(T)|^\frac{1}{2}=|(P_{\lambda}(T))^\ast|^\frac{1}{2}|P_{\lambda}(T)|^\frac{1}{2},
		\end{align*}
		i.e. $P_{\lambda}(T)$ is not binormal. The following result shows that the essential reason why $P_{\lambda}(T)$ fails to be binormal is that
		$$U^\ast|T|U|T^\ast|=\begin{bmatrix}
			20&-\sqrt{3}&0\\
			-5\sqrt{3}&2&0\\
			0&0&1
		\end{bmatrix}\ne\begin{bmatrix}
			20&-5\sqrt{3}&0\\
			-\sqrt{3}&2&0\\
			0&0&1
		\end{bmatrix}=|T^\ast|U^\ast|T|U.$$
	\end{example}
	
	Now we characterize when the power mean transform preserves operator binormality.
	
	\begin{corollary}\label{commu7}
		Let $T\in\mathbb{B}(\mathcal{H})$ be binormal with the polar decomposition $T=U|T|$ and $\lambda\in(0,1)$. Then the following conditions are equivalent:
		\begin{itemize}
			\item[(i)] $P_{\lambda}(T)$ is binormal.
			\item[(ii)] $[U^\ast|T|U,|T^\ast|]=0$.
		\end{itemize}
	\end{corollary}
	\begin{proof}
		From Lemma \ref{commu} and Theorem \ref{commu3} we know that 
		\begin{equation}\label{self}
			UU^\ast|T|^\frac{1}{2}=|T|^\frac{1}{2}UU^\ast,
		\end{equation} 
		\begin{equation}\label{self5}
			U^\ast|T|^\frac{1}{2}U|T|^\frac{1}{2}=|T|^\frac{1}{2}U^\ast|T|^\frac{1}{2}U
		\end{equation}
		and $P_{\lambda}(T)=U|P_{\lambda}(T)|$ is the polar decomposition of $P_{\lambda}(T)$, where
		\begin{equation}\label{positive1}
			|P_{\lambda}(T)|=((1-\lambda)|T|^\frac{1}{2}+\lambda U^\ast|T|^{\frac{1}{2}}U)^2.
		\end{equation}
		
		(i) $\Rightarrow$ (ii): Suppose $P_{\lambda}(T)$ is binormal. Applying Theorem \ref{commu3} (ii) to $P_{\lambda}(T)$ yields that $UU^\ast|P_{\lambda}(T)|^\frac{1}{2}=|P_{\lambda}(T)|^\frac{1}{2}UU^\ast$. Then, together with \eqref{self} and \eqref{positive1}, we get  
		\begin{equation}\label{self7}
			[UU^\ast,U^\ast|T|^\frac{1}{2}U]=0.
		\end{equation}
		Combining $U^\ast[UU^\ast,U^\ast|T|^\frac{1}{2}U]U=0$ with \eqref{self}, \eqref{self5} and $U^\ast UU^\ast=U^\ast$, we have
		\begin{equation}\label{self6}
			[U^\ast U^\ast|T|^\frac{1}{2}UU,U^\ast|T|^\frac{1}{2}U]=0.
		\end{equation}
		On the other hand, applying Theorem \ref{commu3} again to $P_{\lambda}(T)$, we obtain that $U^\ast\Delta(P_{\lambda}(T))$ is self-adjoint, i.e.,$$ U^\ast|P_{\lambda}(T)|^\frac{1}{2}U|P_{\lambda}(T)|^\frac{1}{2}=|P_{\lambda}(T)|^\frac{1}{2}U^\ast|P_{\lambda}(T)|^\frac{1}{2}U$$
		Substituting \eqref{positive1} into the above equation gives 
		\begin{equation}\label{expan}
			\begin{aligned}
				&((1-\lambda)U^\ast|T|^\frac{1}{2}U+\lambda U^\ast U^\ast|T|^\frac{1}{2}UU)((1-\lambda)|T|^\frac{1}{2}+\lambda U^\ast|T|^\frac{1}{2}U)\\&=((1-\lambda)|T|^\frac{1}{2}+\lambda U^\ast|T|^\frac{1}{2}U)((1-\lambda)U^\ast|T|^\frac{1}{2}U+\lambda U^\ast U^\ast|T|^\frac{1}{2}UU).
			\end{aligned}
		\end{equation}
		Substituting \eqref{self5} and \eqref{self6} into the expansion of \eqref{expan}, we get $[U^\ast U^\ast|T|^\frac{1}{2}UU,|T|^\frac{1}{2}]=0$. Then combining $U[U^\ast U^\ast|T|^\frac{1}{2}UU,|T|^\frac{1}{2}]U^\ast=0$ with \eqref{self7} and $UU^\ast U=U$, we have
		$$[U^\ast|T|^\frac{1}{2}U,U|T|^\frac{1}{2}U^\ast]=0.$$
		Together with \eqref{squareroot2} and Lemma \ref{commu}, the above equation yields $[U^\ast|T|U,U|T|U^\ast]=0$, i.e., $[U^\ast|T|U,|T^\ast|]=0$.
		
		(ii) $\Rightarrow$ (i): From $UU^\ast U=U$, \eqref{self} and \eqref{self5} we have
		\begin{align*}
			|T|^\frac{1}{2}U|T|^\frac{1}{2}U^\ast=U(U^\ast|T|^\frac{1}{2}U|T|^\frac{1}{2})U^\ast=U|T|^\frac{1}{2}U^\ast(|T|^\frac{1}{2}UU^\ast)=U|T|^\frac{1}{2}U^\ast|T|^\frac{1}{2},
		\end{align*}
		i.e.,
		\begin{equation}\label{positive2}
			[|T|^\frac{1}{2},U|T|^\frac{1}{2}U^\ast]=0.
		\end{equation}
		Using the equation $UU^\ast=P_{\overline{\mathcal{R}(|T^\ast|)}}=P_{\mathcal{N}(|T^\ast|)^\perp}$ and the condition 
		\begin{equation}\label{commu5}
			[U^\ast|T|U,|T^\ast|]=[U^\ast|T|U,U|T|U^\ast]=0
		\end{equation}together with Lemma \ref{commu8}, we obtain
		$[U^\ast|T|U,UU^\ast]=0$ and thus 
		\begin{equation}\label{commu4}
			[U^\ast|T|^\frac{1}{2}U,UU^\ast]=0.
		\end{equation}
		Therefore, by \eqref{commu4} and \eqref{self5}we have
		\begin{align*}
			(U^\ast|T|^\frac{1}{2}UUU^\ast)|T|^\frac{1}{2}UU^\ast=UU^\ast (U^\ast|T|^\frac{1}{2}U|T|^\frac{1}{2}U)U^\ast&=UU^\ast|T|^\frac{1}{2}(U^\ast|T|^\frac{1}{2}UUU^\ast)\\&=UU^\ast|T|^\frac{1}{2}UU^\ast U^\ast|T|^\frac{1}{2}U,
		\end{align*}
		i.e., 
		\begin{equation}\label{positive3}
			[U^\ast|T|^\frac{1}{2}U,UU^\ast|T|^\frac{1}{2}UU^\ast]=0.
		\end{equation}
		Moreover, from \eqref{self} and \eqref{commu5} we easily obtain 
		\begin{equation}\label{positive4}
			[U^\ast|T|^\frac{1}{2}U,U|T|^\frac{1}{2}U^\ast]=0
		\end{equation}
		and
		\begin{equation}\label{positive5}
			[|T|^\frac{1}{2},UU^\ast|T|^\frac{1}{2}UU^\ast]=0].
		\end{equation}
		Recall that $P_{\lambda}(T)=U|P_{\lambda}(T)|$ is the polar decompositio of $P_{\lambda}(T)$. Then we have
		\begin{equation}\label{positive6}
			|(P_{\lambda}(T))^\ast|^\frac{1}{2}=U|P_{\lambda}(T)|^\frac{1}{2}U^\ast.
		\end{equation}
		Finally, a direct computation using \eqref{positive1}, \eqref{positive6}, \eqref{positive2}, \eqref{positive3}, \eqref{positive4} and \eqref{positive5} yields$$|P_{\lambda}(T)|^\frac{1}{2}|(P_{\lambda}(T))^\ast|^\frac{1}{2}=|(P_{\lambda}(T))^\ast|^\frac{1}{2}|P_{\lambda}(T)|^\frac{1}{2},$$that is, $P_{\lambda}(T)$ is binormal.
		
		The proof is complete.
	\end{proof}

	To characterize when the Duggal transform $P_{1}(T)=T^D$ of a binormal operator $T$ is binormal, we need the theorem below.
	
	\begin{theorem}\cite[Theorem 3.7 and 3.8]{Mat-Ope-2009}\label{matoper}
		Let $T\in\mathbb{B}(\mathcal{H})$ be binormal with the polar decomposition $T=U|T|$. Then $T^D=U^\ast UU|T^D|$ is the polar decomposition of $T^D$ and $[UU^\ast,U^\ast U]=0$, where $U^\ast UU$ is the partial isometry part.
	\end{theorem}
	
	\begin{theorem}
		Let $T\in\mathbb{B}(\mathcal{H})$ be binormal with the polar decomposition $T=U|T|$. Then the Duggal transform $P_{1}(T)=T^D$ is binormal if and only if $[U^\ast|T|U,UU^\ast]=0$.
	\end{theorem}
	\begin{proof}
		From Theorem \ref{commu3} (ii) we know that
		\begin{equation}\label{squareroot4}
			U^\ast|T|^\frac{1}{2}U|T|^\frac{1}{2}=|T|^\frac{1}{2}U^\ast|T|^\frac{1}{2}U
		\end{equation}
		and
		\begin{equation}\label{commu2}
			|T|UU^\ast=UU^\ast|T|.
		\end{equation}
		The above equation together with the fact $U^\ast UU^\ast=U^{\ast}$ and the uniqueness of the positive square root implies
		\begin{equation}\label{squareroot3}
			|T^D|=((T^D)^\ast T^D)^\frac{1}{2}=(U^\ast|T||T|U)^\frac{1}{2}=U^\ast|T|U.
		\end{equation}
		Equations \eqref{squareroot4} and \eqref{squareroot2} gives \begin{equation}\label{commu6}
			U^\ast|T|U|T|=|T|U^\ast|T|U.
		\end{equation}
		On the other hand,	from Theorem \ref{matoper} we know that $T^D=U^\ast UU|T^D|$ is the polar decomposition of $T^D$. Then from \eqref{squareroot3} and Theorem \ref{matoper} we obtain
		\begin{align*}
			|(T^D)^\ast|=(U^\ast UU)|T^D|(U^\ast UU)^\ast&=U^\ast UUU^\ast|T|UU^\ast U^\ast U=U^\ast U|T|UU^\ast=|T|UU^\ast=UU^\ast|T|.
		\end{align*}
		Therefore, we conclude that	$T^D$ is binormal if and only if 
		\begin{align*}
			&|T^D||(T^D)^\ast|=|(T^D)^\ast||T^D|\\&\Leftrightarrow U^\ast|T|UUU^\ast|T|=UU^\ast(|T|U^\ast|T|U)\\&\Leftrightarrow U^\ast|T|UUU^\ast|T|=UU^\ast(U^\ast|T|U|T|)\quad\quad\text{by \eqref{commu6}}.
		\end{align*}
		Thus, it is enough to show that $U^\ast|T|UUU^\ast|T|=UU^\ast U^\ast|T|U|T|\Leftrightarrow[U^\ast|T|U,UU^\ast]=0$. Indeed, suppose $U^\ast|T|UUU^\ast|T|=UU^\ast U^\ast|T|U|T|$ holds, then we have $U^\ast|T|UUU^\ast=UU^\ast U^\ast|T|U$ on $\overline{\mathcal{R}(|T|)}$. Since $P_{\mathcal{N}(|T|)}=1-U^\ast U$, by Theorem \ref{matoper} we obtain
		\begin{align*}
			U^\ast|T|UUU^\ast P_{\mathcal{N}(|T|)}=U^\ast|T|UUU^\ast(1-U^\ast U)&=U^\ast|T|UUU^\ast-U^\ast|T|U(UU^\ast U^\ast U)\\&=U^\ast|T|UUU^\ast-U^\ast|T|U(U^\ast UUU^\ast)\\&=U^\ast|T|UUU^\ast-U^\ast|T|UUU^\ast\\&=0\\&=UU^\ast U^\ast|T|UP_{\mathcal{N}(|T|)}\quad\quad \text{since $\mathcal{N}(U)=\mathcal{N}(|T|)$}.
		\end{align*}
		So we deduce $U^\ast|T|UUU^\ast=UU^\ast U^\ast|T|U$ on $\mathcal{H}=\overline{\mathcal{R}(|T|)}\oplus\mathcal{N}(|T|)$, i.e., $[U^\ast|T|U,UU^\ast]=0$. Conversely, equation $[U^\ast|T|U,UU^\ast]=0$ immediately implies $U^\ast|T|UUU^\ast|T|=UU^\ast U^\ast|T|U|T|$.
		This completes the proof.
	\end{proof}
	\begin{remark}
		Even if $T=U|T|$ is binormal, $[U^\ast|T|U,UU^\ast]=0$ (note that $UU^\ast=P_{\overline{\mathcal{R}(|T^\ast|)}}$) in the above theorem and $[U^\ast|T|U,|T^\ast|]=0$ in Corollary \ref{commu7} are not equivalent. This can be seen by considering the invertible operator $T$ in Example \ref{exmaple1}. Clearly, $[U^\ast|T|U,UU^\ast]=0$ holds. Yet, the last computation in Example \ref{exmaple1} implies that $[U^\ast|T|U,|T^\ast|]\ne0$.
	\end{remark}
	
	Let $\lambda\in(0,1]$ and $T\in\mathbb{B}(\mathcal{H})$. Define the $n$th iteration of the power mean transform of $T$ recursively by$$P_{\lambda}^{(n)}(T)=P_{\lambda}\Big(P_{\lambda}^{(n-1)}\Big)(T),\qquad n\in\mathbb{N}$$ with $P_{\lambda}^{(0)}(T)=T$. At the end of this section, we use the previous results to study the limit of iteration of the power mean transform of centered operators. We need the following properties of centered operators

	\begin{lemma}\label{centerpro}
		Let $T\in\mathbb{B}(\mathcal{H})$ be centered with the polar decomposition $T=U|T|$. Then: \begin{itemize}
			\item[(i)] $[(U^\ast)^i|T|^\frac{1}{2}U^i,U|T|U^\ast]=0$ for $i=0,1,2,\cdots$.
			\item[(ii)] $[(U^\ast)^i|T|^\frac{1}{2}U^i,UU^\ast]=0$ for $i=0,1,2,\cdots$.
			\item[(iii)] 
			$(U^\ast)^i|T|^\frac{1}{2}U^i(U^\ast)^j=(U^\ast)^i|T|^\frac{1}{2} U^{i-j}$ for non‑negative integers $j\le i$.
			\item[(iv)] $(U^\ast)^i|T|^\frac{1}{2}U^i=((U^\ast)^i|T|U^i)^\frac{1}{2}$ for $i=0,1,2,\cdots$.
			\item[(v)] $[(U^\ast)^i|T|^\frac{1}{2}U^i,(U^\ast)^j|T|^\frac{1}{2}U^j]=0$ for $i,j=0,1,2,\cdots$.
		\end{itemize}
	\end{lemma}
	\begin{proof}
	Since $T$ is centered, it follows from (v) and (iv) of Theorem 4.1 in \cite{Ito-JOT-2004} (see also Theorem 4.6 in \cite{Xu-AOT-2019}) that
	\begin{equation}\label{commu11}
	[U^i|T|(U^i)^\ast,|T|]
	\end{equation}
and 
\begin{equation}\label{commu12}
[(U^\ast)^i|T|U^i,(U^\ast)^j|T|U^j]=0
\end{equation}
for $i,j=0,1,2,\cdots$.	

Using \eqref{commu11} together with (v) and (vi) of Lemma 3.7 in \cite{Ito-JOT-2004}, we deduce that (i) and (iii) are valid.

Using the equation $UU^\ast=P_{\overline{\mathcal{R}(|T^\ast|)}}=P_{\mathcal{N}(|T^\ast|)^\perp}$ and (i), noting that $U|T|U^\ast=|T^\ast|$, together with Lemma \ref{commu8}, we obtain (ii).

A direct computation gives\begin{align*}
(U^\ast)^i|T|^\frac{1}{2}U^i(U^\ast)^i|T|^\frac{1}{2}U^i\xlongequal{\text{(iii)}}(U^\ast)^i|T|^\frac{1}{2}U^{i-i}|T|^\frac{1}{2}U^i=(U^\ast)^i|T|U^i
\end{align*}
for $i=0,1,2,\cdots$. From the above equation and the uniqueness of the positive square root, we obtain (iv). 

Combining (iv) and \eqref{commu12} with Lemma \ref{commu} yields (v).

The proof is complete.
\end{proof}
The following result gives an explicit form of the polar decomposition of the $n$th iteration of the power mean transform of a central operator.
\begin{proposition}\label{bin}
 Let $T\in\mathbb{B}(\mathcal{H})$ be centered with the polar decomposition $T=U|T|$ and $\lambda\in(0,1)$. Then $P_{\lambda}^{(k)}(T)$ is binormal and $P_{\lambda}^{(k)}(T)=U|P_{\lambda}^{(k)}(T)|$ is the polar decomposition of $P_{\lambda}^{(k)}(T)$ for $k=1,2,\cdots,n$, where $|P_{\lambda}^{(k)}(T)|=\Big(\sum\limits_{i=0}^k\binom{k}{i}\lambda^i(1-\lambda)^{k-i}(U^\ast)^i|T|^\frac{1}{2}U^i\Big)^2$.
\end{proposition}
\begin{proof}
	 We prove this by induction on $n$. Case $n=1$. Since $T$ is centered, it is also binormal. Then, by Theorem	\ref{commu3} (iv), we obtain the desired result. 
	
	Now assume that the above proposition is true for some positive integer $n$; we shall prove that it is true for $n+1$ as well. Indeed, by the induction hypothesis, i.e., $P_{\lambda}^{(n)}(T)=U|P_{\lambda}^{(n)}(T)|$ is the polar decomposition, by Lemma \ref{centerpro} (ii) we have 
	\begin{align*}
	UU^\ast|P_{\lambda}^{(n)}(T)|^\frac{1}{2}&=\sum\limits_{i=0}^n\binom{n}{i}\lambda^i(1-\lambda)^{n-i}\Big(UU^\ast(U^\ast)^i|T|^\frac{1}{2}U^i\Big)\\&=\sum\limits_{i=0}^n\binom{n}{i}\lambda^i(1-\lambda)^{n-i}\Big((U^\ast)^i|T|^\frac{1}{2}U^iUU^\ast\Big)\\&=|P_{\lambda}^{(n)}(T)|^\frac{1}{2}UU^\ast,
	\end{align*}
	which, by Lemma \ref{commu}, implies
	\begin{equation}\label{itera}
		UU^\ast|P_{\lambda}^{(n)}(T)|=|P_{\lambda}^{(n)}(T)|UU^\ast.
	\end{equation}
	By Lemma \ref{centerpro} (v), we obtain
	\begin{align*}
		&U^\ast\Delta\Big(P_{\lambda}^{(n)}(T)\Big)=U^\ast|P_{\lambda}^{(n)}(T)|^\frac{1}{2}U|P_{\lambda}^{(n)}(T)|^\frac{1}{2}\\&=\Bigg(\sum\limits_{i=0}^n\binom{n}{i}\lambda^i(1-\lambda)^{n-i}(U^\ast)^{i+1}|T|^\frac{1}{2}U^{i+1}\Bigg)\Bigg(\sum\limits_{j=0}^n\binom{n}{j}\lambda^j(1-\lambda)^{n-j}(U^\ast)^{j}|T|^\frac{1}{2}U^{j}\Bigg)\\&=\sum\limits_{0\le i,j\le n}\binom{n}{i}\binom{n}{j}\lambda^{i+j}(1-\lambda)^{2n-i-j}\Big((U^\ast)^{i+1}|T|^\frac{1}{2}U^{i+1}(U^\ast)^{j}|T|^\frac{1}{2}U^{j}\Big)\\&=\sum\limits_{0\le i,j\le n}\binom{n}{i}\binom{n}{j}\lambda^{i+j}(1-\lambda)^{2n-i-j}\Big((U^\ast)^{j}|T|^\frac{1}{2}U^{j}(U^\ast)^{i+1}|T|^\frac{1}{2}U^{i+1}\Big)\\&=\Bigg(\sum\limits_{j=0}^n\binom{n}{j}\lambda^j(1-\lambda)^{n-j}(U^\ast)^{j}|T|^\frac{1}{2}U^{j}\Bigg)\Bigg(\sum\limits_{i=0}^n\binom{n}{i}\lambda^j(1-\lambda)^{n-i}(U^\ast)^{i+1}|T|^\frac{1}{2}U^{i+1}\Bigg)\\&=|P_{\lambda}^{(n)}(T)|^\frac{1}{2}U^\ast|P_{\lambda}^{(n)}(T)|^\frac{1}{2}U=\Bigg(U^\ast\Delta\Big(P_{\lambda}^{(n)}(T)\Big)\Bigg)^\ast.
	\end{align*}
	Combining the above equation and \eqref{itera} with Theorem \ref{commu3} (ii), we deduce that $P_{\lambda}^{(n)}(T)$ is binormal. Therefore, it follows from Theorem \ref{commu3} (iii) that $P_{\lambda}^{(n+1)}(T)=U|P_{\lambda}^{(n+1)}(T)|$ is the polar decomposition of $P_{\lambda}^{(n+1)}(T)$, where
	\begin{equation}\label{itera1}
		|P_{\lambda}^{(n+1)}(T)|=\Big((1-\lambda)|P_{\lambda}^{(n)}(T)|^\frac{1}{2}+\lambda U^\ast|P_{\lambda}^{(n)}(T)|^\frac{1}{2} U\Big)^2.
	\end{equation} To finish the proof, it remain to prove $|P_{\lambda}^{(n+1)}(T)|^\frac{1}{2}=\sum\limits_{i=0}^{n+1}\binom{n+1}{i}\lambda^i(1-\lambda)^{n+1-i}(U^\ast)^i|T|^\frac{1}{2}U^i$. From \eqref{itera1} and the induction hypothesis, we see that\begin{align*}
		&|P_{\lambda}^{(n+1)}(T)|^\frac{1}{2}=(1-\lambda)|P_{\lambda}^{(n)}(T)|^\frac{1}{2}+\lambda U^\ast|P_{\lambda}^{(n)}(T)|^\frac{1}{2} U\\&=(1-\lambda)\sum\limits_{i=0}^n\binom{n}{i}\lambda^i(1-\lambda)^{n-i}(U^\ast)^i|T|^\frac{1}{2}U^i+\lambda\sum\limits_{i=0}^n\binom{n}{i}\lambda^i(1-\lambda)^{n-i}(U^\ast)^{i+1}|T|^\frac{1}{2}U^{i+1}\\&=\sum\limits_{i=0}^n\binom{n}{i}\lambda^i(1-\lambda)^{n+1-i}(U^\ast)^i|T|^\frac{1}{2}U^i+\sum\limits_{i=0}^n\binom{n}{i}\lambda^{i+1}(1-\lambda)^{n-i}(U^\ast)^{i+1}|T|^\frac{1}{2}U^{i+1}\\&=\sum\limits_{i=0}^n\binom{n}{i}\lambda^i(1-\lambda)^{n+1-i}(U^\ast)^i|T|^\frac{1}{2}U^i+\sum\limits_{j=1}^{n+1}\binom{n}{j-1}\lambda^{j}(1-\lambda)^{n+1-j}(U^\ast)^{j}|T|^\frac{1}{2}U^{j}\\&=(1-\lambda)^{n+1}|T|^\frac{1}{2}+\lambda^{n+1}(U^\ast)^{n+1}|T|^\frac{1}{2}U^{n+1}+\sum\limits_{i=1}^{n}\Bigg(\binom{n}{i}+\binom{n}{i-1}\Bigg)\lambda^{i}(1-\lambda)^{n+1-i}(U^\ast)^{i}|T|^\frac{1}{2}U^{i}\\&=(1-\lambda)^{n+1}|T|^\frac{1}{2}+\lambda^{n+1}(U^\ast)^{n+1}|T|^\frac{1}{2}U^{n+1}+\sum\limits_{i=1}^{n}\binom{n+1}{i}\lambda^{i}(1-\lambda)^{n+1-i}(U^\ast)^{i}|T|^\frac{1}{2}U^{i}\\&=\sum\limits_{i=0}^{n+1}\binom{n+1}{i}\lambda^{i}(1-\lambda)^{n+1-i}(U^\ast)^{i}|T|^\frac{1}{2}U^{i}.
	\end{align*}
	The proof is complete.		
\end{proof}

Now, we present the main result, which was proved by Osaka and Yamazaki for invertible $T$ using the operator mean \cite[Theorem 3.3]{yamazaki-tams-2025}; under our definition, however, the invertibility of $T$ is not required.
\begin{theorem}
Let $T\in\mathbb{B}(\mathcal{H})$ be centered with the polar decomposition $T=U|T|$ and $\lambda\in(0,1)$. If $T$ is semi-hyponormal, then $\lim\limits_{n\to\infty}P_{\lambda}^{(n)}(T)=UL$ in the strong operator topology. Moreover, $UL$ is quasinormal and $\|UL\|=\|P_{\lambda}^{(n)}(T)\|$ for $n\in\mathbb{N}$.
\end{theorem}
\begin{proof}
By \eqref{normest2}, Proposition \ref{bin} and Remark \ref{rkbin} (ii), we have
\begin{equation}\label{inseq}
	\|T\|\ge\Big\|P_{\lambda}^{(n+1)}(T)\Big\|=\Big\|\Big|P_{\lambda}^{(n+1)}(T)\Big|\Big\|\ge\Big|P_{\lambda}^{(n+1)}(T)\Big|\ge\Big|P_{\lambda}^{(n)}(T)\Big|,
\end{equation}
which implies that $\{|P_{\lambda}^{(n)}(T)|\}_{n\in\mathbb{N}}$ is a bounded increasing sequence, and hence converges in the strong operator topology to a positive operator $L\in\mathbb{B}(\mathcal{H})$ (\cite[Problem 120]{Halmos-GTM-1982}). From Proposition \ref{bin} we have $P_{\lambda}^{(n)}(T)=U|P_{\lambda}^{(n)}(T)|$; therefore $\lim\limits_{n\to\infty}P_{\lambda}^{(n)}(T)=UL$ in the strong operator topology. 

Since $UU^\ast|P_{\lambda}^{(n)}(T)|=|P_{\lambda}^{(n)}(T)|UU^\ast$ for $n\in\mathbb{N}$, we obtain 
\begin{equation}\label{commu20}
UU^\ast L=LUU^\ast.
\end{equation}From \eqref{itera1}, we get $\Big|P_{\lambda}^{(n+1)}(T)\Big|^\frac{1}{2}=(1-\lambda)|P_{\lambda}^{(n)}(T)|^\frac{1}{2}+\lambda U^\ast|P_{\lambda}^{(n)}(T)|^\frac{1}{2} U$, which yields that $L^\frac{1}{2}=(1-\lambda)L^\frac{1}{2}+\lambda U^\ast L^\frac{1}{2}U$; together with \eqref{commu20} and Lemma \ref{commu}, this implies $UL=LU$.

 Now, we need only prove $\mathcal{N}(U)=\mathcal{N}(L)$. Then $UL$ is the polar decomposition (see \cite[Chapter VIII, \S 3, the Polar Decomposition Theorem]{Conway-book-1990}) and consequently $UL$ is quasinormal. Indeed, from Proposition \ref{kernel} we have $\mathcal{N}(U)=\mathcal{N}(|P_{\lambda}^{(n)}(T)|)\subseteq\mathcal{N}(L)$ for $n\in\mathbb{N}$. By \eqref{inseq}, $L\ge\Big|P_{\lambda}^{(0)}T\Big|=|T|$, hence $\mathcal{N}(L)\subseteq\mathcal{N}(|T|)=\mathcal{N}(U)$. Therefore, $\mathcal{N}(U)=\mathcal{N}(L)$, as desired.
 
  Finally, from \eqref{inseq} we obtain $$\|T\|\ge\|UL\|=\|L\|\ge\|P_{\lambda}^{(n)}(T)\|\ge\cdots\ge\|P_{\lambda}^{(0)}(T)\|=\|T\|,$$ so that $\|UL\|=\|P_{\lambda}^{(n)}(T)\|$ for $n\in\mathbb{N}$. 
\end{proof}
\begin{remark}
Note that for invertible $T$, we have $P_{\lambda}(T)=\Delta_{\mathsf{m}_{f}}(T)$ with $f(x)=(\lambda+(1-\lambda)\sqrt{x})^2$ on $(0,\infty)$. Since the derivative $f'(1)=1-\lambda\in(0,1)$, the condition that $T$ is semi-hyponormal cannot be removed (see \cite[Theorem 5.2]{yamazaki-laa-2021}).
\end{remark}

Recall that a weighted unilateral shift with an increasing weight sequence is centered and semi‑hyponormal. The following example illustrates the above theorem.
\begin{example}
Let $\alpha\equiv\{\alpha_{n}\}_{n\in\mathbb{N}}$ be a bounded increasing sequence of positive numbers and $W_{\alpha}$ the unilateral weighted shift. Let $\lambda\in(0,1)$. Since $\sup\limits_{n\ge0}\alpha_{n}\operatorname{shift}(1,1,\cdots)$ is quasinormal, we only need to prove that
\begin{equation}\label{normcov}
	\lim\limits_{m\to\infty}P_{\lambda}^{(m)}(W_{\alpha})=\sup\limits_{n\ge0}\alpha_{n}\operatorname{shift}(1,1,\cdots)
\end{equation}
 in the norm topology and $\|P_{\lambda}^{(m)}(T)\|=\sup\limits_{n\ge0}\alpha_{n}$ for $m=0,1,2,\cdots$. Indeed, From \eqref{uni} we have$$P_{\lambda}(W_{\alpha})e_{n}=((1-\lambda)\sqrt{\alpha_{n}}+\lambda\sqrt{\alpha_{n+1}})^2e_{n+1},$$where $\{e_{n}\}_{n\in\mathbb{N}}$ is the standard orthonormal basis. It follows by induction on $m$ that 
 \begin{equation}\label{powerin}
 P_{\lambda}^{(m)}(W_{\alpha})e_{n}=\Bigg(\sum\limits_{i=0}^m\binom{m}{i}\lambda^i(1-\lambda)^{m-i}\sqrt{\alpha_{n+i}}\Bigg)^2e_{n+1}.
 \end{equation}
A direct computation gives that $$\|P_{\lambda}^{(m)}(W_{\alpha})\|=\sup\limits_{n\ge0}\Bigg(\sum\limits_{i=0}^m\binom{m}{i}\lambda^i(1-\lambda)^{m-i}\sqrt{\alpha_{n+i}}\Bigg)^2=\sup\limits_{n\ge0}(\sqrt{\alpha_{n}})^2(\lambda+1-\lambda)^{2m}=\sup\limits_{n\ge0}\alpha_{n}.$$
Let us now prove \eqref{normcov}. Since $\{\alpha_{n}\}$ is increasing,
\begin{align*}
0&\le\Big(\sum\limits_{i=0}^m\binom{m}{i}\lambda^i(1-\lambda)^{m-i}(\sup\limits_{n\ge0}\sqrt{\alpha_{n}}-\sqrt{\alpha_{n+i}})\Big)\Big(\sup\limits_{n\ge0}\sqrt{\alpha_{n}}+\sum\limits_{i=0}^m\binom{m}{i}\lambda^i(1-\lambda)^{m-i}\sqrt{\alpha_{n+i}}\Big)\\&=\Big(\sup\limits_{n\ge0}\sqrt{\alpha_{n}}-\sum\limits_{i=0}^m\binom{m}{i}\lambda^i(1-\lambda)^{m-i}\sqrt{\alpha_{n+i}}\Big)\Big(\sup\limits_{n\ge0}\sqrt{\alpha_{n}}+\sum\limits_{i=0}^m\binom{m}{i}\lambda^i(1-\lambda)^{m-i}\sqrt{\alpha_{n+i}}\Big)\\&=\sup\limits_{n\ge0}\alpha_{n}-\Bigg(\sum\limits_{i=0}^m\binom{m}{i}\lambda^i(1-\lambda)^{m-i}\sqrt{\alpha_{n+i}}\Bigg)^2.			
\end{align*}
Together with \eqref{powerin}, we obtain
\begin{small}
	\begin{equation}\label{co1}
	\begin{aligned}
		\Big\|P_{\lambda}^{(m)}(W_{\alpha})-\big(\sup\limits_{n\ge0}\alpha_{n}\big)\operatorname{shift}(1,1,\cdots)\Big\|&=\sup\limits_{n\ge0}\alpha_{n}-\Bigg(\sum\limits_{i=0}^m\binom{m}{i}\lambda^i(1-\lambda)^{m-i}\sqrt{\alpha_{i}}\Bigg)^2\\&\le2\sup\limits_{n\ge0}\sqrt{\alpha_{n}}\cdot\Big(\sup\limits_{n\ge0}\sqrt{\alpha_{n}}-\sum\limits_{i=0}^m\binom{m}{i}\lambda^i(1-\lambda)^{m-i}\sqrt{\alpha_{n+i}}\Big).
	\end{aligned}
	\end{equation}
\end{small}
For $1\le k\le m$,
\begin{small}
\begin{equation}\label{co2}
	0\le\sum\limits_{i=k}^m\binom{m}{i}\lambda^i(1-\lambda)^{m-i}(\sup\limits_{n\ge0}\sqrt{\alpha_{n}}-\sqrt{\alpha_{i}})\le(\sup\limits_{n\ge0}\sqrt{\alpha_{n}}-\sqrt{\alpha_{k}})\sum\limits_{i=k}^m\binom{m}{i}\lambda^i(1-\lambda)^{m-i}\le\sup\limits_{n\ge0}\sqrt{\alpha_{n}}-\sqrt{\alpha_{k}}
\end{equation}
\end{small}
and
\begin{small}
\begin{equation}\label{co3}
	0\le\sum\limits_{i=0}^{k-1}\binom{m}{i}\lambda^i(1-\lambda)^{m-i}(\sup\limits_{n\ge0}\sqrt{\alpha_{n}}-\sqrt{\alpha_{i}})\le(\sup\limits_{n\ge}\sqrt{\alpha_{n}}-\sqrt{\alpha_{0}})\max\{\lambda^m,(1-\lambda)^m\}k\cdot\frac{m!}{(m-k+1)!}.
\end{equation}
\end{small}
Hence, by \eqref{co1}, \eqref{co2} and \eqref{co3}, we have $$\lim\limits_{m\to\infty}\Big\|P_{\lambda}^{(m)}(W_{\alpha})-\big(\sup\limits_{n\ge0}\alpha_{n}\big)\operatorname{shift}(1,1,\cdots)\Big\|=0,$$ as desired. 
\end{example}

\section{Maps commuting with the power mean transform of product of matrices}

Chabbabi in \cite{chabbabi-jmaa-2017} described the form of all bijective maps $\Phi$ on $\mathbb{B}(\mathcal{H})$ that satisfy$$\Delta(\Phi(T)\Phi(S))=\Phi(\Delta(TS))\qquad\text{for all $T,S\in\mathbb{B}(\mathcal{H})$}.$$ Later, in \cite{chabbai-rtoa-2019}, they considered an analogous problem for the mean transform; more precisely, they characterized all bijections $\Phi$ on $\mathbb{B}(\mathcal{H})$ with$$M(\Phi(T)\Phi(S))=\Phi(M(TS))\qquad\text{for all $T,S\in\mathbb{B}(\mathcal{H})$}.$$In view of the above, one may wonder about the form of bijections $\Phi$ on $\mathbb{B}(\mathcal{H})$ such that $$\frac{1}{2}\Delta(\Phi(T)\Phi(S))+\frac{1}{2}M(\Phi(T)\Phi(S))=\Phi(\frac{1}{2}\Delta(TS)+\frac{1}{2}M(TS)),$$ equivalently,
\begin{equation}\label{zeroproduct}
P(\Phi(T)\Phi(S))=\Phi(P(TS))
\end{equation}
 for all $T,S\in\mathbb{B}(\mathcal{H})$. Thanks to Proposition \ref{kernel}, we can then use the recent ideas of Bourhim et al. \cite{bour-lma-2026} to answer this question in the matrix case. More generally, the following theorem is established. We write $T\sim_{u} S$ if there exists a unitary $U\in\mathbb{B}(\mathcal{H})$ such that $T=USU^\ast$. For $T=(t_{ij})\in\mathbb{M}_{n\times m}(\mathbb{C})$ (the set of all $n\times m$ matrices over $\mathbb{C}$), the complex conjugation of $T$ is denoted by $\overline{T}$.
 \begin{theorem}\label{mapmain}
 	Let $\lambda\in(0,1)$. For any bijective map $\Phi$ on $\mathbb{M}_{n}(\mathbb{C})$ with $n\ge2$, the following are equivalent:
 	\begin{itemize}
 		\item[(i)] $\Phi$ satisfies
 		\begin{equation}\label{mapcom1}
 			P_{\lambda}(\Phi(T)\Phi(S))=\Phi(P_{\lambda}(TS))
 		\end{equation}
 		for all $T,S\in\mathbb{M}_{n}(\mathbb{C})$.
 		
 		\item[(ii)] $\Phi$ satisfies
 		\begin{equation}\label{mapcom}
 			P_{\lambda}(\Phi(T)\Phi(S))\sim_{u}\Phi(P_{\lambda}(TS))
 		\end{equation}
 		for all $T,S\in\mathbb{M}_{n}(\mathbb{C})$.
 		\item[(iii)] There exist a unitary matrix $U\in\mathbb{M}_{n}(\mathbb{C})$ such that either \begin{equation*}
 			\Phi(T)=UTU^\ast
 		\end{equation*}
 		or
 		\begin{equation*}
 			\Phi(T)=U\overline{T}U^\ast
 		\end{equation*}
 		for all $T\in\mathbb{M}_{n}(\mathbb{C})$.
 	\end{itemize}
 \end{theorem}
 \begin{remark}
 This theorem is false in one dimension. For instance, $\Phi(z)=\begin{cases}
 	\frac{1}{z}&z\ne0\\
 	0&z=0
 \end{cases}$ 
 is bijective, satisfies \eqref{mapcom1}, but is not additive.
 \end{remark}
 To prove Theorem \ref{mapmain}, some lemmas are first presented.

\begin{lemma}\label{form7}
	Let $\lambda\in(0,1)$ and $T\in\mathbb{B}(\mathcal{H})$. Then
	\begin{itemize}
	\item[(i)] $P_{\lambda}(cT)=cP_{\lambda}(T)$ for all $c\in\mathbb{C}$.
	\item[(ii)] 
	$P_{\lambda}(UTU^\ast)=UP_{\lambda}(T)U^\ast$ for every unitary operator $U\in\mathbb{B}(\mathcal{H})$.
	\item[(iii)] 
	$P_{\lambda}(\overline{T})=\overline{P_{\lambda}(T)}$ if $\mathbb{B}(\mathcal{H})=\mathbb{M}_{n}(\mathbb{C})$.
	\item[(iv)]
	For any nonzero vectors $x,y\in\mathcal{H}$ with $\dim\mathcal{H}\ge2$, we have $$\sigma(P_{\lambda}(x\otimes y))=\sigma(x\otimes y)=\{0,\langle x,y\rangle\}.$$ In particular, $P_{\lambda}(x\otimes x)=x\otimes x$ and $\sigma(P_{\lambda}(x\otimes x))=\{0,\|x\|^2\}$.
	\item[(v)] For nonzero vectors $x,y\in\mathcal{H}$, $x$ and $y$ are linearly dependent if and only if $P_{\lambda}(x\otimes y)\sim_{u}x\otimes y$.
	\end{itemize}
\end{lemma}
\begin{proof}
Let $T=V|T|$ be the polar decomposition of $T$.	
	
(i) One can easily obtain the polar decomposition $cT=e^{i\theta}U|cT|$ with $c=|c|e^{i\theta}$. Then (i) follows from the definition of the power mean transform.
	
(ii) By \cite[Lemma 4.13]{zhou-glma-2023}, we have the polar decomposition $UTU^\ast=(UVU^\ast)(U|T|U^\ast)$ with partial isometry $UVU^\ast$ and $\Delta(UTU^\ast)=U\Delta(T)U^\ast$. Then (ii) follows from the definition of the power mean transform.
	
(iii) It is easy to verify that $\overline{T}=\overline{V}\,\overline{|T|}$ is the polar decomposition of $\overline{T}$, which immediately yields (iii).
	
(iv) $\sigma(x\otimes y)=\{0,\langle x,y\rangle\}$ is well known. From \eqref{rankpower} we obtain $$P_{\lambda}(x\otimes y)=((1-\lambda)^2x+(1-(1-\lambda)^2)\frac{\langle x,y\rangle}{\|y\|^2}y)\otimes y.$$ $\sigma(P_{\lambda}(x\otimes y))=\{0,\langle x,y\rangle\}$ holds due to the above equality and$$\langle((1-\lambda)^2x+(1-(1-\lambda)^2)\frac{\langle x,y\rangle}{\|y\|^2}y),y\rangle=\langle x,y\rangle.$$
	
(v) 
Suppose $P_{\lambda}(x\otimes y)\sim_{u}x\otimes y$. Then from the unitary invariance of the norm and \eqref{normrankone} we have
$$\|(1-\lambda)^2x+(1-(1-\lambda)^2)\frac{\langle x,y\rangle}{\|y\|^2}y\|\|y\|=\|x\|\|y\|,$$
hence $$\|(1-\lambda)^2x+(1-(1-\lambda)^2)\frac{\langle x,y\rangle}{\|y\|^2}y\|=\|x\|.$$ Applying the triangle inequality together with the Cauchy–Schwarz inequality gives$$\|x\|\le(1-\lambda)^2\|x\|+(1-(1-\lambda)^2)\Big\|\frac{\langle x,y\rangle}{\|y\|}\Big\|\le\|x\|.$$
Thus, $|\langle x,y\rangle|=\|x\|\|y\|$, i.e., $x$ and $y$ are linearly dependent. The converse follows directly from (i) and (iv).
\end{proof}

Let $\tau$ be an automorphism of $\mathbb{C}$. For $T=(t_{ij})\in\mathbb{M}_{n\times m}(\mathbb{C})$ and $x=\begin{pmatrix}
	x_{1}\\
	\vdots\\
	x_{n}
\end{pmatrix}\in\mathbb{C}^n$, define $T_{\tau}=(\tau(t_{ij}))$ and $x_{\tau}=\begin{pmatrix}
	\tau(x_{1})\\
	\vdots\\
	\tau(x_{n})
\end{pmatrix}$. Take nonzero vectors $x,y\in\mathbb{C}^n$. The rank-one operator $x\otimes y$ has the matrix representation \(\begin{bmatrix} x_{1}\overline{y_{1}} & \cdots&x_{1}\overline{y_{n}} \\ \vdots & \ddots&\cdots\\x_{n}\overline{y_{1}}&\cdots&x_{n}\overline{y_{n}} \end{bmatrix}\). Using this matrix representation, it is easy to see that 
\begin{equation}\label{rankoneaut}
	(x\otimes y)_{\tau}=x_{\tau}\otimes\overline{\{{\overline{y}\}}_{\tau}}.
\end{equation}
The following two lemmas are essentially due to Bourhim et al. \cite{bour-lma-2026}. For the sake of compatibility, we prove the second lemma below.

\begin{lemma}\cite[p. 81]{bour-lma-2026}\label{auttau}
Let $A$ be an invertible matrix in $\mathbb{M}_{n}(\mathbb{C})$. If $Ax_{\tau}$ and $A^{\ast -1}(\overline{\{\overline{x}\}_{\tau}})$ are linearly dependent for every nonzero $x\in\mathbb{C}^n$, then $A$ is a positive scalar multiple of a unitary matrix $U$ and $\tau$ is either the identity map or the conjugation map on $\mathbb{C}$.
\end{lemma}

Denote by $\mathcal{F}_{1}(\mathbb{C}^n)$ the set of matrices  of rank one in $\mathbb{M}_{n}(\mathbb{C})$. 

\begin{lemma}\label{auteq}
Let $\lambda\in(0,1)$ and let $\alpha:\mathcal{F}_{1}(\mathbb{C}^n)\to\mathbb{C}\backslash\{0\}$ be a function satisfying
\begin{equation}\label{form3}
\alpha(R)\alpha(S)P_{\lambda}(RS)\sim_{u}\alpha(P_{\lambda}(RS))P_{\lambda}(RS)
\end{equation}
 for all $R,S\in\mathcal{F}_{1}(\mathbb{C}^n)$. Then $\alpha(x\otimes y)=\alpha(x\otimes z)$ for all nonzero vectors $x,y,z\in\mathbb{C}^n$.
\end{lemma}
\begin{proof}
{\bf Case 1}: Suppose $y$ and $z$ are linearly independent. Then there exist $a,b\in\mathbb{C}$ such that $\widetilde{x}=ay+bz$ satisfies $\langle y,\widetilde{x}\rangle=\langle z,\widetilde{x}\rangle=1$. Let $A=x\otimes y$, $B=x\otimes z$ and $C=\widetilde{x}\otimes x$. Since $AC=BC=x\otimes x$, by Lemma \ref{form7} (iv) we may substitute $R=A,S=C$ and $R=B,S=C$ into \eqref{form3} to obtain
\begin{equation}\label{form5}
\alpha(A)\alpha(C)(x\otimes x)\sim_{u}\alpha(x\otimes x)(x\otimes x),
\end{equation}
\begin{equation}\label{form6}
\alpha(B)\alpha(C)(x\otimes x)\sim_{u}\alpha(x\otimes x)(x\otimes x).
\end{equation}
Using Lemma \eqref{form7} (iv) and the unitary invariance of the spectrum together with \eqref{form5} and \eqref{form6}, we deduce $$\alpha(A)\alpha(C)=\alpha(x\otimes x)=\alpha(B)\alpha(C).$$
Thus, $\alpha(x\otimes y)=\alpha(A)=\alpha(B)=\alpha(x\otimes z)$ for linearly independent $y,z$.

{\bf Case 2}: Suppose $y$ and $z$ are linearly dependent. Since $\dim(\mathbb{C}^n)\ge2$, there exists a vector $\widetilde{z}$ such that $y$ and $\widetilde{z}$ are linearly independent; then $z$ and $\widetilde{z}$ are also linearly independent. Applying Case 1 we obtain obtain $\alpha(x\otimes y)=\alpha(x\otimes\widetilde{z})=\alpha(x\otimes z)=\alpha(x\otimes z)$.

Hence, $\alpha(x\otimes y)=\alpha(x\otimes z)$ holds for all nonzero vectors $x,y,z\in\mathbb{C}^n$.
\end{proof}

Now we prove Theorem \ref{mapmain}.

\begin{proof}[Proof of Theorem \ref{mapmain}]
(i) $\Rightarrow$ (ii) is obvious. (iii) $\Rightarrow$ (i) is immediate from Lemma \ref{form7} (ii) and (iii). Thus, we only need to prove (ii) $\Rightarrow$ (iii). Now, suppose (ii) holds. Then\\~

{\bf Claim 1}: $\Phi(0)=0$\\~

Take $A\in\mathbb{M}_{n}(\mathbb{C})$ such that $\Phi(A)=0$. For any $S$, by \eqref{mapcom} we have $\Phi(P_{\lambda}(AS))=0$ for all $S\in\mathbb{M}_{n}(\mathbb{C})$. By the injectivity of $\Phi$, we obtain $P_{\lambda}(AS)=A$. Taking $S=0$ gives $P_{\lambda}(0)=A$, so $A=0$. Thus $\Phi(0)=0$.\\~

{\bf Claim 2}: There exist a function $\alpha:\mathcal{F}_{1}(\mathbb{C}^n)\to\mathbb{C}\backslash\{0\}$ and a unitary matrix $U\in\mathbb{M}_{n}(\mathbb{C})$ such that either
\begin{equation*}
	\Phi(R)=\alpha(R)URU^\ast
\end{equation*}
or
\begin{equation*}
	\Phi(R)=\alpha(R)U\overline{R}U^\ast
\end{equation*}
for all rank-one matrices $R\in\mathbb{M}_{n}(\mathbb{C})$.\\~

Proposition \ref{kernel} and Claim 1 readily imply that $\Phi$ preserves zero products, i.e., $\Phi(T)\Phi(S)=0\Leftrightarrow TS=0$ for all $T,S\in\mathbb{M}_{n}(\mathbb{C})$. By \cite[Theorem 3.6]{bour-lma-2026} (see also \cite[Lemma 2.2 (3)]{hou-laa-2008}), there exist a function $\alpha:\mathcal{F}_{1}(\mathbb{C}^n)\to\mathbb{C}\backslash\{0\}$, an invertible matrix $A\in\mathbb{M}_{n}(\mathbb{C})$ and a field automorphism $\tau:\mathbb{C}\to\mathbb{C}$ such that\begin{equation}\label{automap}
\Phi(R)=\alpha(R)AR_{\tau}A^{-1}
\end{equation} for all rank-one matrices $R\in\mathbb{M}_{n}(\mathbb{C})$. 
Take a nonzero $x\in\mathbb{C}^n$ and set $R=x\otimes x$. Then from \eqref{automap}, \eqref{rankoneaut} and Lemma \ref{form7} we obtain
\begin{align*}
\alpha(R)^2\tau(\|x\|^2)P_{\lambda}(AR_{\tau}A^{-1})=P_{\lambda}(\alpha(R)^2AR_{\tau}^2A^{-1})&=P_{\lambda}(\Phi(R)^2)\\&\sim_{u}\Phi(P_{\lambda}(R^2))\\&=\Phi(\|x\|^2P_{\lambda}(R))\\&=\alpha(\|x\|^2R)\tau(\|x\|^2)AR_{\tau}A^{-1}.
\end{align*}
Taking the spectrum on both sides and using the fact that $AR_{\tau}A^{-1}=(Ax_{\tau})\otimes (A^{\ast -1}(\overline{\{\overline{x}\}_{\tau}}))$ with $\langle Ax_{\tau},A^{\ast -1}(\overline{\{\overline{x}\}}_{\tau})\rangle=\langle x_{\tau},\overline{\{\overline{x}\}_{\tau}}\rangle=\sum\limits_{i=1}^n|x_{i}|^2\tau(1)=\|x\|^2$, we deduce from Lemma \ref{form7} (iv) and \eqref{mapcom} that $$\alpha(R)^2\tau(\|x\|^2)=\alpha(\|x\|^2R)\tau(\|x\|^2).$$
Consequently, $P_{\lambda}(AR_{\tau}A^{-1})\sim_{u} AR_{\tau}A^{-1}=(Ax_{\tau})\otimes (A^{\ast -1}(\overline{\{\overline{x}\}_{\tau}}))$. From Lemma \ref{form7} (v) we obtain that $Ax_{\tau}$ and $A^{\ast -1}(\overline{\{\overline{x}\}_{\tau}})$ are linearly dependent. By Lemma \ref{auttau}, $A$ is a positive scalar multiple of a unitary matrix $U$ and $\tau$ is either the identity map or the complex conjugation on $\mathbb{C}$. Together with \eqref{automap}, the claim follows.\\~ 
 
 Define a mapping  $\overline{\Phi}$ by $\overline{\Phi}(T)=\overline{\Phi(T)}$ for all $T\in\mathbb{M}_{n}(\mathbb{C})$. Then by Lemma \ref{form7}, $\overline{\Phi}$ and $U^\ast\Phi(\cdot)U$ also satisfy equation \eqref{mapcom}. Together with Claim 2, we may assume that 
 \begin{equation}\label{basform}
 \Phi(R)=\alpha(R)R
 \end{equation}
  holds for every rank‑one matrix $R\in\mathbb{M}_{n}(\mathbb{C})$. It then remains to show that $\Phi(T)=T$ for all $T\in\mathbb{M}_{n}(\mathbb{C})$.\\~
 
 {\bf Claim 3}: $\alpha(R)=1$ for all rank-one matrices $R\in\mathbb{M}_{n}(\mathbb{C})$.\\~
 
 Take any two rank‑one matrices $R,S\in\mathbb{M}_{n}(\mathbb{C})$. Then from \eqref{basform}, \eqref{mapcom} and Lemma \ref{form7},
 \begin{equation}\label{rsequ}
 \alpha(R)\alpha(S)P_{\lambda}(RS)=P_{\lambda}(\Phi(R)\Phi(S))\sim_{u}\Phi(P_{\lambda}(RS))=\alpha(P_{\lambda}(RS))P_{\lambda}(RS).
 \end{equation}
 Hence, by Lemma \ref{auteq}, we obtain
 \begin{equation}\label{rsequ2}
 \alpha(x\otimes y)=\alpha(x\otimes z)
 \end{equation}
 for all nonzero vectors $x,y,z\in\mathbb{C}^n$. Moreover, set $R=S=P$ in \eqref{rsequ}, where $P$ is a rank‑one projection. Since $P_{\lambda}(P)=P$, we obtain $\alpha(P)^2P\sim_{u}\alpha(P)P$. Taking the spectrum on both sides, yields $\alpha(P)=1$. Now, set $z=\frac{x}{\|x\|^2}$ in \eqref{rsequ2}. Because $x\otimes\frac{x}{\|x\|^2}$ is a rank-one projection, we get $\alpha(x\otimes y)=1$ for all nonzero vectors $x,y\in\mathbb{C}^n$. The claim follows.\\~
 
 {\bf Claim 4}: $\Phi(T)=T$ for all $T\in \mathbb{M}_{n}(\mathbb{C})$.\\~
 
 From Claim 3 and \eqref{basform}, we obtain $\Phi(R)=R$ for all rank-one matrices $R$. Take any unit vector $x\in\mathbb{C}^n$ and any matrix $T\in\mathbb{M}_{n}(\mathbb{C})$. Then by Lemma \ref{form7} (iv) we have
 $$P_{\lambda}(\Phi(T)x\otimes x)=P_{\lambda}(\Phi(T)\Phi(x\otimes x))\sim_{u}\Phi(P_{\lambda}(Tx\otimes x))=P_{\lambda}(Tx\otimes x).$$
 Taking the spectrum on both sides and using Lemma \ref{form7} (iv), we obtain
 $\langle\Phi(T)x,x\rangle=\langle Tx,x\rangle$, which implies $\Phi(T)=T$.\\~
 
 The proof is complete.
 
\section*{Author Contributions} 
All authors contributed equally towards the paper.
 
\section*{Conflicts of Interest}
The authors declare no conflicts of interest.

\section*{Data Availability Statement}
Data sharing not applicable to this article as no datasets were generated or analysed during the current study.

\end{proof}

	\bibliographystyle{amsplain}
	
\end{document}